\documentclass[11pt]{article}

\usepackage{amsmath}
\usepackage{amsfonts}
\usepackage{amsthm}
\usepackage{amssymb}
\usepackage{MnSymbol}
\usepackage{bbm}
\usepackage{cite}
\usepackage{enumerate}
\usepackage[top=1in,bottom=1in,left=1.2in,right=1.2in]{geometry}

\numberwithin{equation}{section}

\newcommand{\R}{\mathbb{R}}
\newcommand{\N}{\mathbb{N}}
\newcommand{\E}{\mathbb{E}}

\newcommand{\calS}{\mathcal{S}}
\newcommand{\inr}[2]{\langle #1, #2 \rangle}

\newcommand{\pdiff}[2]{\frac{\partial #1}{\partial #2}}

\newcommand{\diff}[2]{\frac{d #1}{d #2}}

\newcommand{\diffat}[3]{\left.\frac{d #1}{d #2}\right|_{#3}}
\newcommand{\grad}{\nabla}
\newcommand{\eps}{\epsilon}
\newcommand{\symdiff}{\Delta}

\newcommand{\Inf}{\mathrm{Inf}}
\newcommand{\sgn}{\mathrm{sgn}}
\newcommand{\Stab}{\mathbb{S}}
\DeclareMathOperator{\Cut}{Cut}
\DeclareMathOperator{\MaxCut}{MaxCut}

\let\P\relax
\let\Pr\relax
\DeclareMathOperator{\P}{Pr}
\DeclareMathOperator{\Pr}{Pr}

\newtheorem{theorem}{Theorem}[section]

\newtheorem{conjecture}[theorem]{Conjecture}
\newtheorem{lemma}[theorem]{Lemma}
\newtheorem{corollary}[theorem]{Corollary}
\newtheorem{proposition}[theorem]{Proposition}

\title{Robust Optimality of Gaussian Noise Stability}
\author{Elchanan Mossel\footnote{U.C. Berkeley. Supported by grants NSF (DMS 1106999 ) and ONR (DOD ONR N000141110140)} and Joe Neeman\footnote{U.C. Berkeley. Supported by grants NSF (DMS 1106999 ) and ONR (DOD ONR N000141110140)}}

\begin{document}
\maketitle

\begin{abstract}
 We prove that under the Gaussian measure, half-spaces are
 uniquely the most noise stable sets. We also prove a quantitative
 version of uniqueness, showing that a set which is almost optimally
 noise stable must be close to a half-space.
 This extends a theorem of Borell,
 who proved the same result but without uniqueness, and it also answers
 a question of Ledoux, who asked whether it was possible to prove
 Borell's theorem using a direct semigroup argument.
 Our quantitative uniqueness result has various applications
 in diverse fields.
\end{abstract}

\section{Introduction}

Gaussian stability theory is a rich extension of Gaussian isoperimetric theory. 
As such it connects
numerous areas of mathematics including probability, geometry~\cite{Borell:85}, concentration
and high dimensional phenomena~\cite{ledoux:book}, re-arrangement 
inequalities~\cite{BurchardSchmuckenschlager:01,IsakssonMossel:12}
and more. On the other hand, this theory has recently found
fascinating applications in combinatorics and theoretical computer science.
It was essential in~\cite{MoOdOl:10}
for proving the ``majority is stablest'' conjecture~\cite{Kalai:02,KKMO:04},
the ``it ain't over until it's over'' conjecture~\cite{Kalai:01},
and for establishing the unique games computational hardness~\cite{Khot:02}
of numerous optimization problems including, for example, constraint satisfaction
problems~\cite{Austrin:07b,DiMoRe:06,Raghavendra:08,KKMO:07}.

The standard measure of stability of a set is the probability
that positively correlated standard Gaussian vectors both lie in the set. 
The main result in this area, which is used in all of the applications
mentioned above,
is that half-spaces have optimal stability 
among all sets with a given Gaussian measure.
This fact was originally proved by Borell~\cite{Borell:85}, in a
difficult proof using Ehrhard symmetrization.
Recently, two different proofs of Borell's result have emerged.
First, Isaksson and the first author~\cite{IsakssonMossel:12}
applied some recent advances in
spherical symmetrization~\cite{BurchardSchmuckenschlager:01} 
to give an proof that also generalizes to a problem
involving more than two Gaussian vectors.
Then Kindler and O'Donnell~\cite{KindlerOdonnell:12},
using the sub-additivity idea of Kane~\cite{Kane:11}, 
gave a short and elegant proof,
but only for sets of measure $1/2$ and
for some special values of the correlation.
 
In this paper, we will give a novel proof of Borell's
result. In doing so, we answer a question posed 18 years ago
by Ledoux~\cite{Ledoux:94}, who used semigroup methods to show that
Borell's inequality implies the Gaussian isoperimetric inequality
and then asked whether similar methods could be used to give a short and direct
proof of Borell's inequality.
Moreover, our proof will allow us to strengthen Borell's result and its
discrete applications.
First, we will demonstrate that half-spaces are the {\em unique} optimizers
of Gaussian stability (up to almost sure equality).
Then we will quantify this statement, by showing that if the stability of a set
is close to optimal given its measure, then the set must be close to a half-space. 

The questions of equality and robustness of isoperimetric inequalities
can be rather more subtle than the inequalities themselves.
In the case of the standard Gaussian isoperimetric result, it took
about 25 years from the time the inequality 
was established~\cite{st78,bor75} before the
equality cases were fully characterized~\cite{ck01}
(although the equality cases among sufficiently nice sets were known
earlier~\cite{Ehrhard:86}). Robust versions of the standard Gaussian
isoperimetric result were first established only
recently~\cite{italian11,MosselNeeman:12}. 
Here, for the first time since Borell's original proof~\cite{Borell:85} more than
25 years ago, 
we establish both that half-spaces are the unique maximizers 
{\em and} that a robust version of this statement is also true.

\subsection{Discrete applications}
From our Gaussian results, we derive robust versions of some of the main
discrete applications of Borell's result, including a robust version of the
``majority is stablest'' theorem~\cite{MoOdOl:10}. The ``majority is stablest''
theorem
concerns subsets $A$ of the discrete cube $\{-1, 1\}^n$ with the property
that each coordinate $x_i$ has only a small \emph{influence} on whether
$x \in A$ (see~\cite{MoOdOl:10} for a precise definition);
the theorem says that over all such sets $A$, the ones with that are
most noise stable take the form $\{x: \sum a_i x_i \le b\}$. From the results
we prove here, it is possible to obtain a robust version of this, which says
that any sets $A\subset \{-1, 1\}^n$ with small coordinate influences and
almost optimal noise sensitivity must be close to some set of the
form $\{x: \sum a_i x_i \le b\}$.

A robust form of the ``majority is stablest'' theorem immediately implies
a robust version of the quantitative Arrow theorem. In economics,
Arrow's theorem~\cite{Arrow:50} says that any non-dictatorial election system
between three candidates
which satisfies two natural properties (namely, the ``independence of irrelevant
alternatives'' and ``neutrality'') has a chance of producing a non-rational
outcome. (By non-rational outcome, we mean that there are three candidates,
$A$, $B$ and $C$ say, such that candidate $A$ is preferred to candidate $B$,
$B$ is preferred to $C$ and $C$ is preferred to $A$.)
Kalai~\cite{Kalai:02,Kalai:04} 
showed that if the election system is such that
each voter has only a small influence on the outcome, then the probability
of a non-rational outcome is substantial; moreover, the ``majority
is stablest'' theorem~\cite{MoOdOl:10} implies that the probability of
a non-rational outcome can be minimized by using a simple majority vote
to decide, for each pair of candidates, which one is preferred.
A robust version of the ``majority is stablest'' theorem implies
immediately that (weighted) majority-based
voting methods are essentially the only low-influence
methods that minimizes the probability of a non-rational outcome.

In a different direction, our robust noise stability result
has an application in hardness of approximation, specifically in the
analysis of the well-known Max-Cut optimization problem. The Max-Cut
problem seeks a partition of a graph $G$ into two pieces such that
the number of edges from one piece to the other is maximal. This problem
is NP-hard~\cite{Karp:72} but Goemans and
Williamson~\cite{GoemansWilliamson:95} gave an approximation algorithm
with an approximation ratio of about 0.878. Their algorithm works by embedding
the graph $G$ on a high-dimensional sphere and then cutting
it using a random hyperplane. Feige and Schechtman~\cite{FeigeSchechtman:02}
showed that a random hyperplane is the optimal way to cut this embedded graph;
with our robust noise stability theorem, we can show that any almost-optimal
cutting procedure is almost the same as using a random hyperplane. 
The latter result is derived via a novel isoperimetric result for spheres in high dimensions where
two points are connected if their inner product is exactly some prescribed number $\rho$. 

\subsection{Borell's theorem and a functional variant}
Let $\gamma_n$ be the standard Gaussian measure on $\R^n$.
For $-1 < \rho < 1$ let $X$ and $Y$ be jointly Gaussian
random vectors on $\R^n$, such that $X$ and $Y$ are standard Gaussian
vectors and $\E X_i Y_j = \delta_{ij} \rho$.
We will write $\Pr_\rho$ for the joint probability distribution
of $X$ and $Y$. We will also write $\phi$ for the density of $\gamma_1$
and $\Phi$ for its distribution function:
\begin{align*}
 \phi(x) &= \frac{1}{\sqrt{2\pi}} e^{-x^2/2} \\
 \Phi(x) &= \int_{-\infty}^x \phi(y)\ dy.
\end{align*}

\begin{theorem}[Borell~\cite{Borell:85}]\label{thm:borell}
  For any $0 < \rho < 1$ and any measureable $A_1, A_2 \subset \R^n$,
  \begin{equation}\label{eq:borell}
    \Pr_\rho(X \in A_1, Y \in A_2) \le \Pr_\rho(X \in B_1, Y \in B_2)
  \end{equation}
  where
  \begin{align*}
    B_1 &= \{x \in \R^n: x_1 \le \Phi^{-1}(\gamma_n(A_1))\} \\
    \text{and } B_2 &= \{x \in \R^n: x_1 \le \Phi^{-1}(\gamma_n(A_2))\}
  \end{align*}
  are parallel half-spaces with the same volumes as $A_1$ and $A_2$
  respectively.

  If $-1 < \rho < 0$ then the inequality~\eqref{eq:borell} is
  reversed.
\end{theorem}

Like many other inequalities about sets, Theorem~\ref{thm:borell}
has a functional analogue. To state it,
we define the function
\[
  J(x, y) = J(x, y; \rho) = \Pr_\rho (X_1 \le \Phi^{-1}(x), Y_1 \le \Phi^{-1}(y)).
\]

\begin{theorem}\label{thm:functional-borell}
  For any measurable functions $f, g: \R^n \to [0, 1]$ and any $0 < \rho < 1$,
  \begin{equation}\label{eq:functional-borell}
    \E_\rho J(f(X), g(Y); \rho) \le
    J(\E f, \E g; \rho)
  \end{equation}

  If $-1 < \rho < 0$ then the inequality~\eqref{eq:functional-borell}
  is reversed.
\end{theorem}

To see that Theorem~\ref{thm:functional-borell} generalizes
Theorem~\ref{thm:borell}, consider $f = 1_{A_1}$ and $g = 1_{A_2}$.
Note that $J(0, 0) = J(1, 0) = J(0, 1) = 0$, while $J(1, 1) = 1$.
Thus, $J(f(X), g(Y)) = 1_{X \in A_1, Y \in A_2}$ and so the
left hand side (resp.\ right hand side)
of Theorem~\ref{thm:functional-borell} is the same
as the left hand side (resp.\ right hand side) of
Theorem~\ref{thm:borell}.

In fact, we can also go in the other direction and prove
Theorem~\ref{thm:functional-borell} from Theorem~\ref{thm:borell}:
given $f, g: \R^n \to [0, 1]$,
define $A_1, A_2 \subset \R^{n+1}$ to be the epigraphs of $\Phi^{-1} \circ f$
and $\Phi^{-1} \circ g$ respectively.
It can be easily checked, then, that
\[\E_\rho J(f(X), g(Y); \rho)
  = \Pr_\rho(\tilde X \in A_1, \tilde Y \in A_2)
\]
where $\tilde X$ and $\tilde Y$ are standard Gaussian vectors on $\R^{n+1}$
with $\E \tilde X_i \tilde Y_i = \delta_{ij} \rho$.
On the other hand, $\E f = \gamma_{n+1}(A_1)$ and $\E g = \gamma_{n+1}(A_2)$
and so the definition of $J$ implies that
\[J(\E f, \E g; \rho)
  = \Pr_\rho(\tilde X \in B_1, \tilde Y \in B_2)
\]
where $B_1$ and $B_2$ are parallel half-spaces with the same volumes as
$A_1$ and $A_2$. Thus, Theorem~\ref{thm:borell} in $n+1$ dimensions
implies Theorem~\ref{thm:functional-borell} in $n$ dimensions.

However, we will give a proof of Theorem~\ref{thm:functional-borell}
that does not rely on Theorem~\ref{thm:borell}. We do this for
two reasons: first, we believe that our proof of Theorem~\ref{thm:functional-borell}
is simpler than existing proofs of Theorem~\ref{thm:borell}.
More importantly, our proof of Theorem~\ref{thm:functional-borell} is a good starting point 
for the main results of the paper. In particular, it allows
 us to characterize the cases of equality and near-equality.
As we mentioned earlier, it is not known
how to get such results from existing proofs of
Theorem~\ref{thm:borell}.

\subsection{New results: Equality}
In our first main result, we get a complete characterization of
the functions for which equality in Theorem~\ref{thm:functional-borell}
is attained.

\begin{theorem}\label{thm:equality}
  For any measurable functions $f, g: \R^n \to [0, 1]$ and any
  $-1 < \rho < 1$ with $\rho \ne 0$,
  if equality is attained in~\eqref{eq:functional-borell}
  then there exist $a, b, d \in \R^n$ such that either
  \begin{align*}
    f(x) &= \Phi(\inr{a}{x - b} ) \text{ a.s.} \\
    g(x) &= \Phi(\inr{a}{x - d} ) \text{ a.s.}
  \end{align*}
  or
  \begin{align*}
    f(x) &= 1_{\inr{a}{x - b} \ge 0} \text{ a.s.} \\
    g(x) &= 1_{\inr{a}{x - d} \ge 0} \text{ a.s.}
  \end{align*}
\end{theorem}

In particular, the second case of Theorem~\ref{thm:equality}
implies that if $A_1$ and $A_2$ achieve equality in Theorem~\ref{thm:borell}
then $A_1$ and $A_2$ must be almost surely equal to parallel half-spaces.

\subsection{New results: Robustness}\label{sec:results-robust}

Once we know the cases of equality, the next natural thing to ask is
whether they are robust:
if $f$ and $g$ almost achieve equality in~\eqref{eq:functional-borell}
-- in the sense that $\E_\rho J(f(X), g(Y)) \ge J(\E f, \E g) - \delta$ --
does it follow that $f$ and $g$ must be close to some functions of the form
$\Phi(\inr{a}{x-b})$?
In the case of the Gaussian isoperimetric inequality,
which can be viewed as a limiting form of Borell's theorem,
the question of robustness was first addressed
by Cianchi et al.~\cite{italian11}, who showed that the answer
was ``yes,'' and gave a bound that depended on both $\delta$ and $n$.
The authors~\cite{MosselNeeman:12} then proved a similar result
which had no dependence on $n$, but a worse (logarithmic, instead of polynomial)
dependence on $\delta$.
The arguments we will apply here are similar to those used
in~\cite{MosselNeeman:12}, but with some improvements.
In particular, we establish a result with no dependence on the dimension,
and with a polynomial dependence on $\delta$ (although we suspect
that the exponent is not optimal).

\begin{theorem}\label{thm:2-funcs-robustness}
  For measurable functions $f, g: \R^n \to [0, 1]$, define
  \begin{equation}\label{eq:def_delta}
    \delta = \delta(f, g) =
    J(\E f, \E g) - \E_\rho J(f(X), g(Y))
  \end{equation}
  and let
  \[
    m = m(f, g) = \E f (1-\E f) \E g (1 - \E g).
  \]
  For any $0 < \rho < 1$,
  there exist $0 < c(\rho), C(\rho) < \infty$ such that
  for any $f, g: \R^n \to [0, 1]$
  there exist $a, b, d \in \R^n$ such that
  \begin{align*}
    \E |f(X) - \Phi(\inr{a}{X-b})|
    &\le C(\rho) m^{c(\rho)} \delta^{\frac{1}{4} \frac{(1 - \rho)(1-\rho^2)}{1 +3\rho}} \\
    \E |g(X) - \Phi(\inr{a}{X-d})|
    &\le C(\rho) m^{c(\rho)} \delta^{\frac{1}{4} \frac{(1 - \rho)(1-\rho^2)}{1 +3\rho}}.
  \end{align*}
\end{theorem}

We should mention that a more careful tracking of constants in our
proof would improve the exponent of $\delta$ slightly. However,
this improvement would not bring the exponent above $\frac 14$ and it would not
prevent the exponent from approaching zero as $\rho \to 1$.

Although Theorem~\ref{thm:2-funcs-robustness}
is stated only for $0 < \rho < 1$, the same result
for $-1 < \rho < 0$ follows from certain symmetries. Indeed, one can easily
check from the definition of $J$ that
$J(x, y; \rho) = x - J(x, 1-y; -\rho)$. Taking expectations,
\begin{align*}
  \E_\rho J(f(X), g(Y); \rho)
  &= \E f - \E_\rho J(f(X), 1-g(Y); -\rho) \\
  &= \E f - \E_{-\rho} J(f(X), 1-g(-Y); -\rho).
\end{align*}

Now, suppose that $-1 < \rho < 0$ and that $f, g$ almost attain
equality in Theorem~\ref{thm:functional-borell}:
\[
  \E_\rho J(f(X), g(Y); \rho) \le J(\E f, \E g; \rho) + \delta.
\]
Setting $\tilde g(y) = 1 - g(-y)$, this implies that
\[
  \E_{-\rho} J(f(X), \tilde g(Y); -\rho) \ge J(\E f, \E \tilde g; -\rho) - \delta.
\]
Since $0 < -\rho < 1$, we can apply
Theorem~\ref{thm:2-funcs-robustness} to $f$ and $\tilde g$ to conclude
that $f$ and $\tilde g$ are close to the equality cases of Theorem~\ref{thm:equality},
and it follows that
$f$ and $g$ are also close to one of these equality cases.
Therefore, we will concentrate for the rest of this article on the
case $0 < \rho < 1$.

\subsection{Optimal dependence on $\rho$ in the case $f = g$}

The dependence on $\rho$ in Theorem~\ref{thm:2-funcs-robustness}
is particularly interesting as $\rho \to 1$, since it is in that limit
that Borell's inequality recovers the Gaussian isoperimetric inequality.
As it is stated, however, Theorem~\ref{thm:2-funcs-robustness} does
not recover a robust version of the Gaussian isoperimetric inequality
because of its poor dependence on $\rho$ as $\rho \to 1$. In particular,
as $\rho \to 1$, the constant $C(\rho, \epsilon)$ grows to infinity,
and the exponent of $\delta$ tends to zero.

It turns out that this poor dependence on $\rho$ is necessary in some sense.
For example, let 
\begin{align*}
f(x) &= 1_{\{x_1 < 0\}} \\
g(x) &= 1_{\{x_1 < 2, x_2 < 0 \text{ or } x_1 < 1, x_2 \ge 0\}}.
\end{align*}
Then
\[
\Pr_\rho(f(X) = 1, g(X) = 0) \le \Pr_\rho(X_1 < 0, Y_1 \ge 1)
\le \exp\Big(\frac{1}{2(1-\rho^2)}\Big),
\]
which tends to zero very quickly as $\rho \to 1$.
In particular, this means that as $\rho \to 1$, $\delta(f, g)$
tends to zero exponentially fast even though $g$ is a constant distance
away from a half-space. Thus, the constant $C(\rho, \epsilon)$ must
blow up as $\rho \to 1$. Similarly, if we redefine $g$ as
\[
 g(x) = 1_{\{x_1 \le 1+O(\delta), x_2 < 0 \text{ or } x_1 < 1 - O(\delta), x_2 \ge 0\}}
\]
then we see that the exponent of $\delta$ in Theorem~\ref{thm:2-funcs-robustness}
must tend to zero as $\rho \to 1$.

We can, however, avoid examples like the above if we restrict
to the case $f = g$. In this case, it turns out that
$\delta(f, f)$ grows only like $(1-\rho)^{-1/2}$ as $\rho \to 1$,
which is exactly the right rate for recovering the Gaussian
isoperimetric inequality.

\begin{theorem}\label{thm:1-func-robustness}
  For every $\epsilon > 0$, there is a $\rho_0 < 1$
  and a $C(\epsilon)$
  such that for any $\rho_0 < \rho < 1$
  and any $f: \R^n \to [0, 1]$ with $\E f = 1/2$,
  there exists $a \in \R^n$ such that
  \[
    \E |f(X) - \Phi(\inr{a}{X})| \le C(\epsilon)
    \Big(\frac{\delta(f, f)}{\sqrt{1-\rho}}\Big)^{\frac{1}{4}-\epsilon}.
  \]
\end{theorem}

The requirement $\E f = 1/2$ is there for technical reasons, and we do not
believe that it is necessary (see Conjecture~\ref{conj:monotonicity}).

By applying Ledoux's result~\cite{Ledoux:96} connecting Borell's inequality
with the Gaussian isoperimetric inequality, Theorem~\ref{thm:1-func-robustness}
has the following corollary (for the definition of Gaussian surface area,
see~\cite{MosselNeeman:12}):

\begin{corollary}\label{cor:iso}
 For every $\epsilon > 0$, there is a $C(\epsilon) < \infty$ such that
 for every set $A \subset \R^n$ such that $\P(A) = 1/2$
 and $A$ has Gaussian surface area less
 than $\frac{1}{\sqrt{2\pi}} + \delta$, there is a half-space $B$
 such that
 \[
  \P(A \symdiff B) \le C(\epsilon) \delta^{1/4 - \epsilon}.
 \]
\end{corollary}

This should be compared with the work of Cianchi et al.~\cite{CFMP:11}, who
gave the best possible dependence on $\delta$, but suffered some unspecified
dependence on $n$:

\begin{theorem}\label{thm:italian}
 For every $n$ and every $a \in (0, 1)$, there is a constant $C(n, a)$
 such that for every set $A \subset \R^n$ such that $\P(A) = a$
 and $A$ has Gaussian surface area less than $\phi(\Phi^{-1}(a)) + \delta$,
 there is a half-space $B$ such that
 \[
  \P(A \symdiff B) \le C(n, a) \delta^{1/2}.
 \]
\end{theorem}

Note that Theorem~\ref{thm:italian} is stronger than Corollary~\ref{cor:iso}
in two senses, but weaker in one. Theorem~\ref{thm:italian}
is stronger since it applies to sets of all volumes and because it has
a better dependence on $\delta$ (in fact, Cianchi et al.\ show that
$\delta^{1/2}$ is the best possible dependence on $\delta$). However,
Corollary~\ref{cor:iso} is stronger in the sense that it -- like the
rest of our robustness results -- has no dependence on the dimension.
For the applications we have in mind, this dimension independence
is more important than having optimal rates. Nevertheless, we conjecture
that it is possible to have both at the same time:

\begin{conjecture}
There are constants $0 < c, C < \infty$ such that for every $A \subset \R^n$
with Gaussian surface area less than $\phi(\Phi^{-1}(\P(A))) + \delta$,
there is a half-space $B$ such that
\[
 \P(A \symdiff B) \le C \P(A)^c \delta^{1/2}.
\]
\end{conjecture}

\subsection{On highly correlated functions}

Let us mention one more corollary of Theorem~\ref{thm:1-func-robustness}.
We have used
$\E_\rho J(f(X), f(Y))$ as a functional generalization of
$\Pr_\rho(X \in A, Y \in A)$. However, $\E_\rho f(X) f(Y)$ is
another commonly used
functional generalization of $\P_\rho(X \in A, Y \in A)$ which appeared, for
example, in~\cite{Ledoux:96}. Since $xy \le J(x, y)$ for $0 < \rho < 1$,
we see immediately that Theorem~\ref{thm:functional-borell}
holds when the left hand side is replaced by $\E_\rho f(X) f(Y)$.
The equality case, however, turns out to be different:
whereas equality in Theorem~\ref{thm:functional-borell}
holds for $f(x) = \Phi(\inr{a}{x-b})$,
there is equality in
\begin{equation}\label{eq:correlation-ineq}
 \E_\rho f(X) f(Y) \le J(\E f, \E f; \rho)
\end{equation}
only when $f$ is the indicator of a half-space.
Moreover, a robustness result for~\eqref{eq:correlation-ineq}
follows fairly easily from Theorems~\ref{thm:2-funcs-robustness}
and~\ref{thm:1-func-robustness}.

\begin{corollary}\label{cor:1-func-robustness}
For any $0 < \rho < 1$, there is a constant $C(\rho) < \infty$
such that if $f: \R^n \to [0, 1]$
satisfies
$\E f = 1/2$ and
\[
 \E f(X) f(Y) \ge \frac{1}{4} + \frac{1}{2\pi} \arcsin(\rho) - \delta
\]
then there is a half-space $B$ such that
\[
 \E |f(X) - 1_B(X)| \le C(\rho) \delta^c,
\]
where $c > 0$ is a universal constant.
\end{corollary}

\subsection{Discrete applications}
Corollary~\ref{cor:1-func-robustness} implies a robust
version of the ``majority is stablest'' theorem~\cite{MoOdOl:10},
which concerns functions of low \emph{influence}
and high noise stability;
for a function $f: \{-1, 1\}^n \to \{-1, 1\}$, we define the
influence of the $i$th coordinate by
\[
 \Inf_i(f) = \P(f(x_1, \dots, x_n) \neq f(x_1, \dots, x_{i-1},
-x_i, x_{i+1}, \dots, x_n))
\]
and the noise stability of $f$ by
\[
 \Stab_\rho(f) = \E_\rho f(\xi) f(\sigma)
\]
where $(\xi,\sigma) = ((\xi_1,\ldots,\xi_n),(\sigma_1,\ldots,\sigma_n)) \in \{-1,1\}^n \times \{-1,1\}^n$ 
is chosen so that $(\xi_i,\sigma_i) \in \{-1,1\}^2$ are independent
random variables with $\E\xi_i = \E\sigma_i = 0$ and $\E_\rho \xi_i \sigma_i = \rho$.

The majority is stablest theorem~\cite{MoOdOl:10} informally 
states that low-influence, balanced functions cannot be
essentially more noise-stable than the majority function.
This was first explicitly conjectured by Khot, Kindler, Mossel, and
O'Don\-nell~\cite{KKMO:07} in a paper studying the hardness of approximation of Max-Cut. 
It was used to show that
approximating the maximum cut in a graph to within a factor
of about $0.87856$ is unique-games hard.
This result is optimal, since the famous 
efficient algorithm of Goemans and
Williamson~\cite{GoemansWilliamson:95} is guaranteed to find
a cut that is within a $0.87856$ factor of the maximum cut.
A special case of the majority is stablest theorem was conjectured
earlier by Kalai~\cite{Kalai:02} in the context of his quantitative version
of Arrow's theorem.

Combining our Gaussian results with the original proof from~\cite{MoOdOl:10},
we obtain a robust version of the majority is stablest theorem:
\begin{theorem}\label{thm:robust-mist}
For every $\delta > 0$, there is a $\tau > 0$ such that the following holds:
suppose that $f: \{-1, 1\}^n \to [0, 1]$ is a function with $\Inf_i(f) \le \tau$ for every $i$.
Then for every $0 < \rho < 1$,
\begin{equation}\label{eq:mist}
\Stab_\rho(f) \le J(\E f, \E f; \rho) + \delta.
\end{equation}
If, moreover, there is some $0 < \rho < 1$ such that
\begin{equation}\label{eq:robust-mist}
 \Stab_\rho(f) \ge J(\E f, \E f; \rho) - \delta
\end{equation}
then there exist $a, b \in \R^n$ such that
\begin{align*}
\E |f(\xi) - 1_{\{\inr{a}{\xi-b} \ge 0\}}| \le C (\rho)\delta^c,
\end{align*}
where $c, C > 0$ are universal constants.
\end{theorem}

If we set $a_n = \frac{1}{\sqrt n} (1, \dots, 1)$ and
$b_n = \Phi^{-1}(\E f) a_n$, then the central limit theorem implies
that $\E 1_{\{\inr{a_n}{\xi-b_n} \ge 0\}} \to \E f$
and $\Stab_\rho(1_{\{\inr{a_n}{\xi-b_n} \ge 0\}}) \to J(\E f, \E f; \rho)$.
In the case $\E f = \frac 12$ and $b_n = 0$,~\eqref{eq:mist} says, therefore,
that no low-influence function can be much more noise
stable than the simple majority function -- this is the content of the majority
is stablest theorem from~\cite{MoOdOl:10}.
Our contribution is~\eqref{eq:robust-mist}, which says that the only low-influence functions
which come close to this bound are close to weighted majority functions.

We remark that Theorem~\ref{thm:robust-mist} is not the most general possible
theorem that we can prove.
In particular, we could state a two-function version of
Theorem~\ref{thm:robust-mist} or
a version that uses the functional $\E_\rho J(f(\xi), f(\sigma); \rho)$ in place of $\Stab_\rho(f)$.
All of these variations, however, are proved in essentially the same way, namely by
combining the ideas from~\cite{MoOdOl:10} with the appropriate Gaussian robustness result.
In order to avoid repetition, therefore, we will only state and prove one version.

\subsection{Spherical noise stability and the Max-Cut problem}

The well-known similarity between a Gaussian vector
and a uniformly random vector on a high-dimensional sphere suggests that there might be a spherical analogue of our Gaussian noise sensitivity result. The correlation structure on the sphere that is most useful is
the uniform measure over all pairs of points $(x, y)$ whose inner product
$\inr{x}{y}$ is exactly $\rho$.
Under this model of noise, we can use robust Gaussian noise sensitivity to show, asymptotically in the dimension, robustness for spherical noise
sensitivity. This uses the theory of spherical harmonics and has applications to rounding semidefinite programs
(in particular, the Goemans-William\-son algorithm for Max-Cut). Our proof
uses and generalizes the work of Klartag and Regev~\cite{KlartagRegev:11}, in which a related problem was studied in the context of 
one-way communication complexity. 

Our spherical noise stability result mostly follows from
Theorem~\ref{thm:2-funcs-robustness}, by replacing $X$ and $Y$ by
$X/|X|$ and $Y/|Y|$. When $n$ is large, these renormalized Gaussian
vectors are uniformly distributed on the sphere and their inner product
is tightly concentrated around $\rho$.
The fact that their inner product is not \emph{exactly} $\rho$ causes
some difficulty, particularly because $Q_\rho$ is actually
orthogonal to the joint distribution of two normalized Gaussians.
Working through this difficulty with some properties of
spherical harmonics, we obtain the following spherical analogue of Theorem~\ref{thm:2-funcs-robustness}:

\begin{theorem}\label{thm:sphere}
  Let $0 < \rho < 1$ and 
  write $Q_{\rho}$ for the measure of $(X,Y)$ on the sphere $S^{n-1}$ where the pair
  $(X,Y)$ is uniformly distributed in
  \[
   \{(x, y) \in S^{n-1} \times S^{n-1}: \inr{x}{y} = \rho\}.
  \]
  For measurable $A_1, A_2 \subset S^{n-1}$, define
  \[
    \delta = \delta(A_1, A_2) =
    Q_\rho(X \in B_1, Y \in B_2)
    - Q_\rho(X \in A_1, Y \in A_2),
  \]
  where $B_1$ and $B_2$ are parallel spherical caps with the same volumes
  as $A_1$ and $A_2$ respectively. Define also
  \[
   m(A_1, A_2) = p(1-p) q(1-q)
  \]
  where $p = \Pr(X \in A_1)$ and $q = \Pr(Y \in A_2)$.
  
  For any $A_1, A_2 \subset S^{n-1}$, there exist parallel
  spherical caps $B_1$ and $B_2$ such that
  \begin{align*}
    Q(A_1 \symdiff B_1)
    &\le C(\rho) m^{c(\rho)} \delta_\ast^{\frac{1}{4} \frac{(1 - \rho)(1-\rho^2)}{1 +3\rho}} \\
    Q(A_2 \symdiff B_2)
    &\le C(\rho) m^{c(\rho)} \delta_\ast^{\frac{1}{4} \frac{(1 - \rho)(1-\rho^2)}{1+3\rho}}.
  \end{align*}
  where $\delta_{\ast} = \max(\delta, n^{-1/2} \log n)$. 
\end{theorem}

The case $\rho=0$ of the above theorem is related to work by Klartag and Regev~\cite{KlartagRegev:11}.
In this case one expects that $X$ and $Y$ should behave as independent random variables on $S^{n-1}$ and that therefore 
for {\em all} $A_1, A_2$,  $Q_0(X \in A_1, Y \in A_2)$ should be close to 
$Q(X \in A_1) Q (Y \in A_2)$. Indeed the main technical statement of Klartag and Regev (Theorem 5.2) says that for every two sets,
\[
| Q_0(X \in A_1, Y \in A_2) - Q(X \in A_1) Q(Y \in A_2) | \leq \frac{C}{n}.
\]
In other words the results of Klartag and Regev show that in the case $\rho=0$, a uniform orthogonal pair $(X,Y)$
on the sphere behaves like a pair of independent random variables up to an error
of order $n^{-1}$, while our results show that for $0 < \rho < 1$, 
$(X,Y)$ that are $\rho$ correlated behave like Gaussians with the same correlation. 

That spherical caps minimize the quantity $Q_\rho(X \in A_1, Y \in A_2)$ over all sets
$A_1$ and $A_2$ with some prescribed volumes is originally due to Baernstein
and Taylor~\cite{BaernsteinTaylor:76}, while a similar result for a different noise model
is due to Beckner~\cite{Beckner:92}. Their results do not follow from ours
because of the dependence on $n$ in Theorem~\ref{thm:sphere}, and so one
could ask for a sharper version of Theorem~\ref{thm:sphere} that does imply
these earlier results. One obstacle is that we do not know a proof
of Beckner's inequality that gives control of the deficit.

\subsubsection{Rounding the Goemans-Williamson algorithm}
Let $G = (V, E)$ be a graph and recall that the Max-Cut problem is to
find a set $A \subset V$ such that the number of edges between $A$
and $V \setminus A$ is maximal. It is of course equivalent to look
for a function $f: V \to \{-1, 1\}$ such that
$\sum_{(u,v) \in E} |f(u) - f(v)|^2$ is maximal. Goemans' and Williamson's
breakthrough was to realize that this combinatorial optimization problem
can be efficiently solved if we relax the range $\{-1, 1\}$ to
$S^{n-1}$. Let us say, therefore, that an embedding $f$ of a graph $G = (V, E)$
into the sphere $S^{n-1}$ is \emph{optimal} if
\[\sum_{(u, v) \in E)} |f(u) - f(v)|^2\]
is maximal. An oblivious rounding procedure is a (possibly random) function
$R: S^{n-1} \to \{-1, 1\}$ (we call it ``oblivious'' because it does not
look at the graph $G$). We will then denote by $\Cut(G, R)$ the expected value
of the cut produced by rounding the worst possible optimal
spherical embedding of $G$:
\[
 \Cut(G, R) = \frac{1}{2} \min_f \E \sum_{(u,v) \in E} |R(f(u)) - R(f(v))|,
\]
where the minimum is over all optimal embeddings $f$.
If $\MaxCut$ denotes the maximum cut in $G$, then Goemans and Williamson~\cite{GoemansWilliamson:95}
showed that when $R(x) = \sgn(\inr{X}{x})$ for a standard Gaussian vector $X$, then
for every graph $G$,
\[
 \Cut(G, R) \ge \MaxCut(G) \min_\theta \alpha_\theta,
\]
where $\alpha_\theta = \frac 2\pi \frac{\theta}{1-\cos \theta}$.
In the other direction, Feige and Schechtman~\cite{FeigeSchechtman:02}
showed that for every oblivious rounding scheme $R$ and every $\epsilon > 0$, there is a
graph $G$ such that
\[
 \Cut(G, R) \le \MaxCut(G) \Big(\epsilon + \min_\theta \alpha_\theta \Big).
\]
In other words, no rounding scheme is better than the half-space rounding
scheme. Using Theorem~\ref{thm:2-funcs-robustness}, we can go further:

\begin{theorem}\label{thm:optimal-rounding}
Suppose $R$ is a rounding scheme on $S^{n-1}$ such that
for every graph $G$ with $n$ vertices,
\[
 \Cut(G, R) \ge \MaxCut(G) \Big( \min_\theta \alpha_\theta - \epsilon\Big).
\]
Then there is a hyperplane rounding scheme $\tilde R$ such that
\[
 \E |R(Y) - \tilde R(Y)| \le C \epsilon_\star^c,
\]
where $Y$ is a uniform (independent of $R$ and $\tilde R$) random vector on $S^{n-1}$,
$C$ and $c$ are absolute constants, and $\epsilon_\star = \max\{\epsilon, n^{-1/2} \log n\}$.
\end{theorem}
In other words, any rounding scheme that is almost optimal is essentially
the same as rounding by a random half-space. 

\subsection{Testing half-spaces}
We quickly sketch an application of
Theorems~\ref{thm:2-funcs-robustness} and~\ref{thm:robust-mist}
to testing. Suppose we are given oracle access to a set $A \subset \R^n$ (meaning
that we are not given an explicit representation of the set, but we can query whether
points belong to $A$),
and we want to design an algorithm that (1) will answer ``yes'' with high probability
if $A$ is a half space and (2) will answer ``no'' with high probability if
$\Pr(A \symdiff B) > \eps$ for all half-spaces $B$.

An efficient test for this problem was found in~\cite{MORS:09}. 
We note that Theorem~\ref{thm:1-func-robustness} provides a simpler and
very direct test just by sampling $\epsilon^{-4-\epsilon}$
pairs $(X_i,Y_i)$ and counting the number of times that $X_i \in A$
and the number of times that $1_A(X_i) = 1_A(Y_i)$. By doing so, we obtain accurate
estimates of $\Pr(A)$ and $\Pr(X \in A, Y \in A)$ and so by Theorem~\ref{thm:1-func-robustness},
we can tell whether $A$ is close to a half-space.

By Theorem~\ref{thm:robust-mist}, this algorithm also applies to linear threshold functions
with low influences on the discrete cube (such functions are called regular in~\cite{MORS:09}).
(By the more general arguments in~\cite{MoOdOl:10}, the algorithm also applies to other
discrete
spaces such as half-spaces in biased cubes or cubes of the form $[q]^n$ for some $q \geq 3$.)
Using the arguments of~\cite{MORS:09} it is then possible to extend the testing
algorithm to general linear
threshold functions on the discrete cube.

\subsection{Proof Techniques} 

\subsubsection{Borell's theorem} 
We prove Theorem~\ref{thm:functional-borell} by differentiating
along the Ornstein-Uhlenbeck
semigroup. This technique was used by Bakry and Ledoux~\cite{BakryLedoux:96}
in their proof of the Gaussian isoperimetric inequality
and, more generally, a Gaussian version of the L\'evy-Gromov comparison theorem.
Recall that the Ornstein-Uhlenbeck semigroup can be specified by
defining, for every $t \ge 0$, the operator
\begin{equation}\label{eq:P_t}
  (P_t f)(x) = \int_{\R^n} f(e^{-t} x + \sqrt{1 - e^{-2t}} y)\: d\gamma_n(y).
\end{equation}
Note that $P_t f \to f$ as $t \to 0$ (pointwise, and also in $L^p$),
while $P_t f \to \E f$ as $t \to \infty$.

Let $f_t = P_t f$, $g_t = P_t g$, and consider the quantity
\begin{equation}\label{eq:R_t-def}
  R_t := \E_\rho J(f_t(X), g_t(Y)).
\end{equation}
As $t \to 0$, $R_t$ converges to the right hand side of~\eqref{eq:functional-borell};
as $t \to \infty$, $R_t$ converges to the left hand side of~\eqref{eq:functional-borell}.
We will prove Theorem~\ref{thm:functional-borell}
by showing that $\diff{R_t}{t} \ge 0$ for all $t > 0$.

\subsubsection{The equality case}
The equality case almost comes for free from our proof of Theorem~\ref{thm:functional-borell}.
Indeed, Lemma~\ref{lem:diff-R_t} writes $\diff{R_t}{t}$ as the expectation of
a strictly positive quantity times
\[|(\grad (\Phi^{-1} \circ f_t))(X) - (\grad (\Phi^{-1} \circ g_t))(Y)|,\]
where $|\cdot|$ denotes the Euclidean norm.
Now, if there is equality in Theorem~\ref{thm:functional-borell} then
$\diff{R_t}{t}$ must be zero for all $t$, which implies that the expression
above must be zero almost surely.
This implies that $\grad (\Phi^{-1} \circ f_t)$
and $\grad (\Phi^{-1} \circ g_t)$ are almost surely equal to the same constant,
and therefore $f_t$ and $g_t$ can be written as $\Phi$ composed with a linear function.
We can then infer the same statement for $f$ and $g$ because $P_t$ is
one-to-one.
 
 \subsubsection{Robustness} 
Our approach to robustness begins similarly to the approach in our recent
work~\cite{MosselNeeman:12}. 
If $\delta(f, g)$ is small then $\diff{R_t}{t}$ must also be small for most $t > 0$.
Looking at the expression in Lemma~\ref{lem:diff-R_t} we first concentrate on the main term: 
$| \grad v_t(X) - \grad w_t(Y)|^2$ where $v_t = \Phi^{-1} \circ f_t$
and $w_t = \Phi^{-1} \circ g_t$. Using an analogue of Poincar\'e's inequality,
we argue that if
the expected value of $|\grad v_t(X) - \grad w_t(Y)|^2$ is small then
$v_t$ and $w_t$ are close to linear functions. 
 
Considerable effort goes into controlling the ``secondary terms''
of the expression in Lemma~\ref{lem:diff-R_t}.
This control is established in a sequence of analytic results,
which rely heavily on the smoothness of the semigroup $P_t$, concentration
of Gaussian vectors and $L^p$ interpolation inequalities.
In the end, we show that if $\delta = \delta(f, g)$ is small then
for every $t > 0$,
$v_t$ is $\epsilon(\delta, t)$ close to a linear function.
Since $\Phi$ is a contraction, this implies that $f_t$ must be close to
a function of the form $\Phi(\inr{x}{a} - b)$.

We would like to then conclude the proof by applying $P_t^{-1}$,
and saying that $f$ must be close to $P_t^{-1} \Phi(\inr{x}{a} - b)$,
which also has the form $\Phi(\inr{x}{a'} - b')$. The obvious problem here is that
$P_t^{-1}$ is not a bounded operator, but we work around this by arguing that it acts
boundedly on the functions that we care about. This part of the argument marks
a substantial departure from~\cite{MosselNeeman:12}, where
our argument used smoothness
and spectral information. Here, we will use a geometric argument
to say that if $h = 1_A - 1_B$ where $B$ is a half-space, then $\E |h|$ can be bounded
in terms of $\E |P_t h|$. This improved argument is essentially the reason
that the rates in Theorem~\ref{thm:2-funcs-robustness} are polynomial,
while the rates in~\cite{MosselNeeman:12} were logarithmic.

\subsection{Subsequent work}

A quite different study of the functional $\E_\rho J(f(X), g(Y); \rho)$ turns out
to yield yet another proof of Borell's inequality: in a subsequent work with
De~\cite{DeMoNe:13}, the authors give a proof of Borell's inequality by
first proving a four-point inequality for $J$ which tensorizes to the discrete cube.
Applying the central limit theorem then recovers Borell's inequality.
That approach is similar to Bobkov's elementary proof of the Gaussian isoperimetric
inequality~\cite{Bobkov:96b}. The proof in~\cite{DeMoNe:13} has an advantage
and a disadvantage compared to the one presented here. The advantage of
the tensorization argument is that
it directly yields some interesting inequalities on the cube (in particular, one
obtains a direct proof of the ``majority is stablest'' theorem), while the
proof we present here has the advantage of giving control over the
deficit. In particular, we don't know how to prove Theorem~\ref{thm:2-funcs-robustness}
using the techniques in~\cite{DeMoNe:13}.

\section{Proof of Borell's theorem}
Recall the definition of $P_t$ and $R_t$ from~\eqref{eq:P_t} and~\eqref{eq:R_t-def}.
In this section, we will compute $\diff{R_t}{t}$ and show that it is non-negative,
thereby proving Theorem~\ref{thm:functional-borell}.
First, define
$v_t = \Phi^{-1} \circ f_t$, $w_t = \Phi^{-1} \circ g_t$, and
$K(x, y; \rho) = \Pr_\rho(X \le x, Y \le b)$. Then
\[
  J(f_t(X), g_t(Y)) = K(v_t(X), w_t(Y)).
\]

\begin{lemma}\label{lem:K-diff}
  \begin{align*}
    \pdiff{K(x, y)}{x} &= \phi(x)
    \Phi\Big(\frac{y - \rho x}{\sqrt{1-\rho^2}}\Big) \\
    \pdiff{K(x, y)}{y} &= \phi(y)
    \Phi\Big(\frac{x - \rho y}{\sqrt{1-\rho^2}}\Big).
  \end{align*}
\end{lemma}

\begin{proof}
  Note that $Y$ can be written as $\rho X + \sqrt{1 - \rho^2} Z$, where
  $X$ and $Z$ independent standard Gaussian vectors.
  Then $\{X \le x, Y \le y\} = \{X \le x, Z \le \frac{y - \rho X}{\sqrt{1 - \rho^2}}\}$,
  and so
  \[
    K(x,y)
    =
    \int_{-\infty}^x \int_{-\infty}^{\frac{y - \rho s}{\sqrt{1-\rho^2}}}
    \phi(s)\phi(t)\: dt\: ds.
  \]
  Differentiating in $x$,
  \begin{align*}
    \pdiff{K(x, y)}{x} &= 
    \int_{-\infty}^{\frac{y - \rho x}{\sqrt{1-\rho^2}}}
    \phi(x)\phi(t)\: dt \\
    &= \phi(x)
        \Phi\Big(\frac{y - \rho x}{\sqrt{1 - \rho^2}}\Big).
  \end{align*}
  This proves the first claim. The second claim follows because $K(x, y)$
  is symmetric in $x$ and $y$.
\end{proof}

\begin{lemma}\label{lem:diff-R_t}
  \[
    \diff{R_t}{t} = \frac{\rho}{2\pi\sqrt{1-\rho^2}}
    \E_\rho \exp\Big(-\frac{v_t^2 + w_t^2 - 2\rho v_t w_t}{2(1-\rho^2)}\Big)
    |\grad v_t - \grad w_t|^2.
  \]
\end{lemma}

Before we prove Lemma~\ref{lem:diff-R_t}, note that it immediately
implies Theorem~\ref{thm:functional-borell} because the right hand
side in Lemma~\ref{lem:diff-R_t} is clearly non-negative.

\begin{proof}
  Set $L = \Delta - \inr{x}{\grad}$; it is
  well-known (and easy to check by direct computation) that $\diff{f_t}{t} = L f_t$
  for all $t \ge 0$.
  The integration by parts formula
  \begin{equation}\label{eq:int-parts}
    \E f(X) L g(X) = -\E \inr{\grad f(X)}{\grad g(X)}
  \end{equation}
  for bounded smooth functions $f$ and $g$
  is also standard and easily checked. Thus,
  \begin{equation}\label{eq:R_t}
    \diff{R_t}{t} = \E_\rho \Big(K_x(v_t(X), w_t(Y)) \diff{v_t(X)}{t} \Big)
    + \E_\rho \Big(K_y(v_t(X), w_t(Y)) \diff{w_t(X)}{t}\Big).
  \end{equation}
  Now, the chain rule implies that $\diff{v_t}{t} = \frac{L f_t}{\phi(v_t)}$.
  Hence, the first term of~\eqref{eq:R_t} is
  \begin{equation}\label{eq:func-borell-1}
    \E_\rho \Big(\frac{K_x(v_t(X), w_t(Y)}{\phi(v_t(X))} L f_t(X) \Big)
    = \E_\rho \Phi\Big(\frac{w_t(Y) - \rho v_t(X)}{\sqrt{1-\rho^2}}\Big) L f_t(X),
  \end{equation}
  where we have used Lemma~\ref{lem:K-diff}. Now write
  $Y = \rho X + \sqrt{1 - \rho^2} Z$ (with $X$ and $Z$ independent);
  conditioning on $Z$ and and applying
  the integration by parts~\eqref{eq:int-parts} with respect to $X$, we have
  \begin{align}
    \eqref{eq:func-borell-1}
    &= -\frac{\rho}{\sqrt{1-\rho^2}}
    \E_\rho \phi\Big(\frac{w_t - \rho v_t}{\sqrt{1-\rho^2}}\Big)
     \inr{\grad w_t - \grad v_t}{\grad f_t} \notag \\
     &= \frac{\rho}{\sqrt{1-\rho^2}}
    \E_\rho \phi\Big(\frac{v_t - \rho w_t}{\sqrt{1-\rho^2}}\Big) \phi(v_t)
     \inr{\grad v_t - \grad w_t}{\grad v_t}.
     \label{eq:func-borell-2}
  \end{align}
  where we have written, for brevity, $v_t$ and $w_t$ instead of $v_t(X)$
  and $w_t(Y)$.
  Since $K$ is symmetric in its arguments, there is a similar computation
  for the second term of~\eqref{eq:R_t}:
  \begin{equation}
    \E \Big(K_y(v_t(X), w_t(Y)) \diff{w_t(X)}{t}\Big) 
    = -\frac{\rho}{\sqrt{1-\rho^2}}
    \E_\rho \phi\Big(\frac{w_t - \rho v_t}{\sqrt{1-\rho^2}}\Big) \phi(w_t)
     \inr{\grad v_t - \grad w_t}{\grad w_t}.
    \label{eq:func-borell-3}
  \end{equation}
  Note that
  \[
    \phi\Big(\frac{w_t - \rho v_t}{\sqrt{1-\rho^2}}\Big) \phi(v_t)
    = \phi\Big(\frac{v_t - \rho w_t}{\sqrt{1-\rho^2}}\Big) \phi(w_t)
    = \frac{1}{2\pi} \exp\Big(-\frac{v_t^2 + w_t^2 - 2\rho v_t w_t}{2(1-\rho^2)}
    \Big);
  \]
  hence, we can plug~\eqref{eq:func-borell-2} and~\eqref{eq:func-borell-3}
  into~\eqref{eq:R_t} to obtain
  \[
    \diff{R_t}{t} = \frac{\rho}{2\pi\sqrt{1-\rho^2}}
    \E \exp\Big(-\frac{v_t^2 + w_t^2 - 2\rho v_t w_t}{2(1-\rho^2)}\Big)
    |\grad v_t - \grad w_t|^2.
    \qedhere
  \]
\end{proof}

\section{The equality case}

Lemma~\ref{lem:diff-R_t} allows us to
analyze the the equality case (Theorem~\ref{thm:equality}), with
very little additional effort. Similar ideas were used by
Carlen and Kerce~\cite{CarlenKerce:01} to analyze the equality case
in the standard Gaussian isoperimetric problem. 
Clearly, Lemma~\ref{lem:diff-R_t} implies that if for every $t$,
$v_t$ and
$w_t$ are linear functions with the same slope, then equality is
attained in Theorem~\ref{thm:functional-borell}. To prove
Theorem~\ref{thm:equality}, we will show that
the converse also holds
(ie.\ if equality is attained
then $v_t$ and $w_t$ are linear functions with the same slope).
Then we will take $t \to 0$ to obtain
the desired conclusion regarding $f$ and $g$.

First of all, if $f(x) = 1_{\inr{a}{x-b} \ge 0}$, then
a direct computation gives
\begin{equation}\label{eq:P_t-of-half-space}
  f_t(x)
  = \Phi\left(k_t \frac{\inr{a}{x - e^t b}}{|a|}\right),
\end{equation}
where $k_t = (e^{2t} - 1)^{-1/2}$.
Since $P_t$ is injective, it follows that whenever
$f_t = \Phi(\inr{a}{x - b'})$ for some $a, b$ with $|a| = k_t$,
$f$ must have the form $f(x) = 1_{\{\inr{a}{x-b} \ge 0\}}$.
Since, moreover, $k_t$ is decreasing in $t$, we have the following lemma:
\begin{lemma}\label{lem:pull-back}
  If $f_t(x) =  \Phi(\inr{a}{x - b'})$ for some $a, b' \in \R^n$ with
  $|a| \le k_t$, then there exists $b \in \R^n$
  such that if $\tilde f(x) = 1_{\{\inr{a}{x-b} \ge 0\}}$ then
  $f = P_s \tilde f$, where $s$ solves $|a| = k_{s+t}$.
\end{lemma}

In order to apply Lemma~\ref{lem:pull-back}, we will use the
following pointwise bound on $\grad v_t$, whose proof can be
found in~\cite{BakryLedoux:96}.
Note that the bound is sharp because,
according to~\eqref{eq:P_t-of-half-space}, equality is attained when
$f$ is the indicator function of a half-space.
\begin{lemma}\label{lem:grad-bound-v}
For any function $f: \R^n \to [0, 1]$, any $t > 0$,
and any $x \in \R^n$,
\[
  |\grad v_t(x)| \le k_t.
\]
\end{lemma}

\begin{proof}[Proof of Theorem~\ref{thm:equality}]
  Suppose that equality is attained in~\eqref{eq:functional-borell}. Since
  $\diff{R_t}{t}$ is non-negative, it must be zero for almost every $t > 0$.
  In particular, we may fix some $t > 0$ such that $\diff{R_t}{t} = 0$.
  Note that everything in Lemma~\ref{lem:diff-R_t} is strictly positive,
  except for the last term, which can be zero. Therefore,
  $\diff{R_t}{t} = 0$ implies that $\grad v_t(X) = \grad w_t(Y)$ almost
  surely. Since the conditional distribution of $Y$ given $X$ is fully supported,
  $\grad v_t$ and $\grad w_t$ must be almost surely equal to some constant $a' \in \R^n$.
  Moreover, $v_t$ and $w_t$ are smooth functions (because $f_t$, $g_t$
  and $\Phi^{-1}$ are smooth); hence, $v_t(x) = \inr{a}{x - b'}$
  and $w_t(x) = \inr{a}{x - d'}$ for some $b', d' \in \R^n$, and so
  \begin{align*}
    f_t(x) &= \Phi(\inr{a}{x - b'}) \\
    g_t(x) &= \Phi(\inr{a}{x - d'}).
  \end{align*}

  Now, Lemma~\ref{lem:grad-bound-v} asserts that $|a| = |\grad v_t| \le k_t$.
  Hence, Lemma~\ref{lem:pull-back} implies that there is some
  $b$ such that if $\tilde f(x) = 1_{\inr{a}{x-b} \ge 0}$ then
  $f = P_s \tilde f$, where $s$ solves $|a| = k_{s+t}$.
  In particular, $f$ takes one of the two forms indicated in
  Theorem~\ref{thm:equality}: if $s = 0$ then
  $f(x) = \tilde f(x) = 1_{\inr{a}{x-b} \ge 0}$.
  On the other hand, $s > 0$ implies, by~\eqref{eq:P_t-of-half-space},
  that $f_s = \Phi(k_s \inr{\frac{a}{|a|}}{x-e^s b})$, which we can
  write in the form $\Phi(\inr{a}{x-b})$ by replacing
  $k_s \frac{a}{|a|}$ with $a$ and $k_s e^s b$ with $b$.
  We complete the proof by applying the same argument to $g$.
\end{proof}

\section{Robustness: approximation for large $t$}

The proof of Theorem~\ref{thm:2-funcs-robustness} follows the same general lines
as the one in~\cite{MosselNeeman:12}. Our starting point is
Lemma~\ref{lem:diff-R_t}, and the observation that
if~\eqref{eq:functional-borell} is close to an equality then
$\diff{R_t}{t}$ must be small for most $t$. 
For such $t$, using Lemma~\ref{lem:diff-R_t}, we will argue that $v_t$
must be close
to linear for that $t$; it then follows
that $f_t$ must be close to one of the equality cases
in Theorem~\ref{thm:equality}.
Finally, we use a time-reversal argument to show that
$f$ must be close to one of those equality cases also.

Our proof will be divided into two main parts. In this section, we will
show that $v_t$ is close to linear; we will give the time-reversal
argument in Section~\ref{sec:time-rev}. The main result in this section,
therefore, is Proposition~\ref{prop:2-funcs-large-t}, which says that
$f_t$ must be close to one of the equality cases of Theorem~\ref{thm:equality}.
Recall the definition of $\delta$ from~\eqref{eq:def_delta}, and
recall that $k_t = (e^{2t} - 1)^{1/2}$.

\begin{proposition}\label{prop:2-funcs-large-t}
  For any $0 < \rho < 1$, and
  for any $t > 0$, there exists $C(t, \rho)$ such that
  for any $f, g$ and for any $0 < \alpha < 1$, there exist $b, d \in \R$
  and $a \in \R^n$ with $|a| \le k_t$ such that
  \begin{multline*}
    \E \big(f_{t}(X) - \Phi(\inr{a}{X} - b)\big)^2
    + \E \big(g_{t}(X) - \Phi(\inr{a}{X} - d)\big)^2 \\
    \le C(t, \rho)
    m(f, g)^{\frac{(1-\rho)^2}{8k_t^2(1 + k_t^2)^2(1+\alpha)}}
    \Big(\frac{\delta}{\alpha}\Big)^{\frac{1}{1 + 4 k_t^2/(1-\rho)}
    \frac{1}{1+\alpha}}
  \end{multline*}
  where $m(f, g) = \E f(1-\E f) \E g (1-\E g)$.
\end{proposition}

Let us observe -- and this will be important when we apply
Proposition~\ref{prop:2-funcs-large-t} -- that by Lemma~\ref{lem:pull-back},
$|a| \le k_t$ implies
that $\Phi(\inr{a}{\cdot} - b)$ can be written in the form
$P_{t+s} 1_B$ for some $s > 0$ and some half-space $B$.

The main goal of this section is to prove
Proposition~\ref{prop:2-funcs-large-t}.
The proof proceeds according to the following steps:
\begin{itemize}
\item  First, using a Poincar\'e-like inequality
  (Lemma~\ref{lem:2-funcs-poincare}) we show 
that if $\E_\rho |\grad v(X) - \grad w(Y)|^2$ is small then $v$ and
$w$ are close to linear functions (with the same slope). 
\item In Proposition~\ref{prop:2-funcs-holder},
  we use the reverse H\"older inequality and some concentration properties
  to show that if $\diff{R_t}{t}$ is small, then
  $\E_\rho |\grad v_t(X) - \grad w_t(Y)|^{2p}$ must be small for some $p<1$. 
\item Using Lemma~\ref{lem:grad-bound-v},
  we argue that if
  $\E_\rho |\grad v_t(X) - \grad w_t(Y)|^{2p}$ is small then
  $\E_\rho |\grad v_t(X) - \grad w_t(Y)|^{2}$ is also small.
  Thus, we can apply the Poincar\'e inequality mentioned in the first
  bullet point, and so we obtain linear approximations for $v_t$ and $w_t$.
\end{itemize}

\subsection{A Poincar\'e-like inequality}
Recall that we proved the equality case by
arguing that if $\diff{R_t}{t} = 0$
then $|\grad v_t(X) - \grad w_t(Y)|$ is identically zero, so
$\grad v_t$ and $\grad w_t$ must be constant and thus $v_t$
and $w_t$ must be linear.
The first step towards a robustness
result is to show that if $|\grad v_t(X) - \grad w_t(Y)|$ is small,
then $v_t$ and $w_t$ must be almost linear, and with the same slope.

\begin{lemma}\label{lem:2-funcs-poincare}
  For any smooth functions $v, w \in L_2(\R^n, \gamma_n)$,
  if we set $a = \frac{1}{2} (\E \grad v + \E \grad w)$
  then for any $0 < \rho < 1$,
  \[
    \E (v(X) - \inr{X}{a} - \E v)^2
    + \E (w(X) - \inr{X}{a} - \E w)^2
    \le \frac{\E_\rho |\grad v(X) - \grad w(Y)|^2}{2(1-\rho)}.
  \]
\end{lemma}

We remark that Lemma~\ref{lem:2-funcs-poincare} achieves equality when
$v$ and $w$ are quadratic polynomials which differ only in the constant term.

In order to prove Lemma~\ref{lem:2-funcs-poincare}, we recall the Hermite
polynomials:
for $k \in \N$, define
$H_k(x) = (k!)^{-1/2} e^{x^2/2} \diff{{}^k}{x^k} e^{-x^2/2}$. It is well-known
that the $H_k$ form an orthonormal
basis of $L_2(\R, \gamma_1)$.
For a multiindex $\alpha \in \N^n$, let
\[
H_\alpha(x) = \prod_{i=1}^n H_{\alpha_i}(x_i).
\]
Then the $H_\alpha$ form an orthonormal basis of $L^2(\R^n, \gamma_n)$.
Define $|\alpha| = \sum_i \alpha_i$;
note that $H_\alpha$ is linear if and only if $|\alpha| = 1$, and
$\alpha_i = 0$ implies that $\pdiff{}{x_i} H_\alpha = 0$.
If $\alpha_i > 0$, define $S_i \alpha$ by
$(S_i \alpha)_i = \alpha_i - 1$
and $(S_i \alpha)_j = \alpha_j$ for $j \ne i$. Then a well-known
recurrence for Hermite polynomials states that
\[
\pdiff{}{x_i} H_\alpha =
\begin{cases}
\sqrt {\alpha_i} H_{S_i \alpha} & \text{if $\alpha_i > 0$} \\
0 & \text{if $\alpha_i = 0$.}
\end{cases}
\]
In particular,
\begin{equation}\label{eq:diff-hermite}
\E \Big(\pdiff{}{x_i} H_\alpha\Big)^2 =
\alpha_i.
\end{equation}

It will be convenient for us to reparametrize the Ornstein-Uhlenbeck
semigroup $P_t$: for $0 < \rho < 1$, let $T_\rho = P_{\log(1/\rho)}$.
It is then easily checked that for any $v \in L^1(\R^n, \gamma_n)$,
$\E_\rho (v(Y) | X) = (T_\rho v)(X)$.

The final piece of background that we need before proving
Lemma~\ref{lem:2-funcs-poincare} is the fact that $T_\rho$ acts diagonally
on the Hermite basis, with
\begin{equation}\label{eq:T_rho-hermite}
  T_\rho H_\alpha = \rho^{|\alpha|} H_\alpha.
\end{equation}

\begin{proof}[Proof of Lemma~\ref{lem:2-funcs-poincare}]
  First, consider two arbitrary functions
  $b(x), c(x) \in L_2(\R^n, \gamma_n)$ and suppose that
  their expansions in the Hermite basis are
  $b = \sum_{\alpha} b_\alpha H_\alpha$ and
  $c = \sum_\alpha c_\alpha H_\alpha$.
  Then
  \begin{align*}
    \E_\rho(b(X) - c(Y))^2
    &= \E b^2 + \E c^2 - 2 \E_\rho b(X) c(Y) \\
    &= \E b^2 + \E c^2 - 2 \E b(X) (T_\rho c)(X) \\
    &= \sum_\alpha \big(b_\alpha^2 + c_\alpha^2 -2 \rho^{|\alpha|}
    b_\alpha c_\alpha\big),
  \end{align*}
  where we have used~\eqref{eq:T_rho-hermite} in the last line to compute
  the Hermite expansion of $T_\rho c$. Now,
  $2 b_\alpha c_\alpha \le b_\alpha^2 + c_\alpha^2$ and so
  \begin{align}
    \E_\rho(b(X) - c(Y))^2
    &= (b_0 - c_0)^2 + \sum_{|\alpha| \ge 1}
    \big(b_\alpha^2 + c_\alpha^2 -2 \rho^{|\alpha|} b_\alpha c_\alpha\big) \notag \\
    &\ge (b_0 - c_0)^2 + \sum_{|\alpha| \ge 1} (b_\alpha^2 + c_\alpha^2)(1-\rho^{|\alpha|}) \notag \\
    &\ge (b_0 - c_0)^2 + (1-\rho) \sum_{|\alpha| \ge 1} \big(b_\alpha^2 + c_\alpha^2\big). \label{eq:hermite-inequality}
  \end{align}

  Now write $v$ and $w$ in the Hermite basis as
  $v = \sum v_\alpha H_\alpha$ and $w = \sum w_\alpha H_\alpha$. Then,
  by~\eqref{eq:diff-hermite},
  \begin{align*}
    \pdiff{v}{x_i} &= \sum_{\alpha_i \ge 1} v_\alpha \sqrt{\alpha_i} H_{S_i \alpha}\\
    \pdiff{w}{x_i} &= \sum_{\alpha_i \ge 1} w_\alpha \sqrt{\alpha_i} H_{S_i \alpha}.
  \end{align*}
  In particular, if we set
  $b = \pdiff{v}{x_i}$, then
  $b_{S_i \alpha} = \sqrt{\alpha_i} v_\alpha$ for any $\alpha$ with $\alpha_i \ge 1$.
  Specifically,
  $b_0 = v_{e_i}$ (where $e_i$ is the multi-index
  with $1$ in position $i$ and 0 elsewhere) and
  \[
    \sum_{|\alpha| \ge 1} b_\alpha^2
    = \sum_{|\alpha| \ge 2, \alpha_i \ge 1} b_{S_i \alpha}^2
    = \sum_{|\alpha| \ge 2, \alpha_i \ge 1} \alpha_i v_\alpha^2
  \]
  (Setting $c = \pdiff{w}{x_i}$,
  there is of course an analogous inequality for $c$ and $w$.)
  Applying this to~\eqref{eq:hermite-inequality}, we have
  \begin{equation}\label{eq:poincare-one-partial-diff}
    \E_\rho\left(\pdiff{v}{x_i}(X) - \pdiff{w}{x_i}(Y)\right)^2
    \ge (v_{e_i} - w_{e_i})^2 + (1-\rho)
    \sum_{|\alpha| \ge 2, \alpha_i \ge 1} \alpha_i (v_\alpha^2 + w_\alpha^2).
  \end{equation}
  Now if we apply~\eqref{eq:poincare-one-partial-diff} for each $i$
  and sum the resulting inequalities, we obtain
  \begin{equation}\label{eq:poincare-whole-grad}
    \E_\rho |\grad v(X) - \grad w(Y)|^2
    \ge \sum_{|\alpha| = 1} (v_\alpha - w_\alpha)^2
    + 2(1-\rho) \sum_{|\alpha| \ge 2} v_\alpha^2 + w_\alpha^2.
  \end{equation}

  On the other hand, let $a = \frac{1}{2}(\E \grad v + \E\grad w)$.
  Since $\E \pdiff{v}{x_i} = v_{e_i}$ and $H_{e_i}(x) = x_i$, it
  follows that
  \[
    \inr{x}{a} = \frac{1}{2} \sum_{|\alpha| = 1} (v_\alpha + w_\alpha) H_\alpha(x).
  \]
  Since $\E v = v_0$, we have
  \[
    \E (v(X) - \inr{X}{a} - \E v)^2 =
    \sum_{|\alpha| = 1} \Big(\frac{v_\alpha - w_\alpha}{2}\Big)^2
    + \sum_{|\alpha| \ge 2} v_\alpha^2.
  \]
  Adding to this the analogous expression for $w$, we obtain
  \begin{multline*}
    2(1-\rho) \big(\E (v(X) - \inr{X}{a} - \E v)^2 +
    \E (w(X) - \inr{X}{a} - \E w)^2\big) \\ =
  (1-\rho) \sum_{|\alpha| = 1} (v_\alpha - w_\alpha)^2
    + 2(1-\rho)\sum_{|\alpha| \ge 2} v_\alpha^2 + w_\alpha^2.
  \end{multline*}
  Noting that $1 - \rho \le 1$, we see that this is
  smaller than~\eqref{eq:poincare-whole-grad}. Hence
  \[
    \big(\E (v(X) - \inr{X}{a} - \E v)^2 +
    \E (w(X) - \inr{X}{a} - \E w)^2\big)
    \le \frac{\E_\rho |\grad v(X) - \grad w(Y)|^2}{2(1-\rho)}.
    \qedhere
  \]
  
\end{proof}

\subsection{A lower bound on $\diff{R_t}{t}$}

Recall the formula for $\diff{R_t}{t}$ given in Lemma~\ref{lem:diff-R_t}.
In this section, we will use the reverse-H\"older inequality to split
this formula into an exponential term and a term depending on
$|\grad v_t(X) - \grad w_t(X)|$. We will then use the smoothness
of $v_t$ and $w_t$ to bound the exponential term, with the following
result:

\begin{proposition}\label{prop:2-funcs-holder}
  For any $0 < \rho < 1$
  and any $t > 0$,
  there is a $c(t, \rho) > 0$ such that
  for any $r \le \frac{1}{1 + 4k_t^2/(1-\rho)}$ and
  for any $f$ and $g$,
  \[
    \diff{R_t}{t}
    \ge c(t, \rho)
    m^{2\frac{k_t^2 (1+k_t)^2}{1-\rho}}
    \big(\E |\grad v_t(X) - \grad w_t(Y)|^{2r}\big)^{1/r}.
  \]
\end{proposition}

There are three main ingredients in the proof of
Proposition~\ref{prop:2-funcs-holder}. The first is the reverse-H\"older
inequality, which states that for any functions $f > 0$ and $g \ge 0$
and for any $\beta > 0$ and $0 < r < 1$ with
$\frac{1}{r} - \frac{1}{\beta} = 1$,
\begin{equation}\label{eq:rev-holder}
  \E fg \ge \big(\E f^{-\beta}\big)^{-1/\beta} \big(\E g^r\big)^{1/r}.
\end{equation}

The second ingredient involves concentration properties of
the Gaussian measure. The proof is a standard computation, and we omit it.

\begin{lemma}\label{lem:lipschitz-gaussian}
  If $f: \R^n \to \R$ is 1-Lipschitz with median $M$ then
  for any $\lambda < 1$,
  \[
    \E \exp(\lambda f^2(X)/2) \le \frac{2}{\sqrt{1-\lambda}}
    e^{\frac{\lambda}{2(1-\lambda)} M^2}.
  \]
\end{lemma}

The third and final ingredient is a relationship between the mean
of $f$ and the median of $v_t$.

\begin{lemma}\label{lem:mean-median}
  If $N_t$ is a median of $v_t$ then
  \[
    m(f) = \E f (1 - \E f) \le 2\exp\Big(
      -\frac{N_t^2}{2 (1 + k_t)^2}
    \Big).
  \]
\end{lemma}

\begin{proof}[Proof of Lemma~\ref{lem:mean-median}]
  Lemma 3.8 of~\cite{MosselNeeman:12} proved that
  if $M_t$ is a median of $f_t$ then
  \[
    \E f \le 2M_t^{\left(\frac{1}{1+k_t}\right)^2}.
  \]
  Recall that $f_t = \Phi \circ v_t$
  and so $M_t = \Phi(N_t)$. Suppose first that $N_t \le 0$. Since
  $\Phi(-x) \le e^{-x^2/2}$, we see that
  $M_t \le e^{-N_t^2/2}$ and so
  \begin{equation}\label{eq:mean-median}
    \E f \le 2\exp\Big(
      -\frac{N_t^2}{2 (1 + k_t)^2}
    \Big).
  \end{equation}
  On the other hand, if $N_t > 0$, we apply the preceding argument
  to $1-f$ and we conclude that
  \begin{equation}\label{eq:mean-median-1-f}
    \E (1-f) \le 2\exp\Big(
      -\frac{N_t^2}{2 (1 + k_t)^2}
    \Big).
  \end{equation}
  Of course, $\max\{\E f, 1 - \E f\} \le 1$ and so we can
  combine~\eqref{eq:mean-median} and~\eqref{eq:mean-median-1-f} to
  prove the second claim of the lemma.
\end{proof}

\begin{proof}[Proof of Proposition~\ref{prop:2-funcs-holder}]
  We begin by applying the reverse-H\"older inequality~\eqref{eq:rev-holder}
  to the equation in Lemma~\ref{lem:diff-R_t}:
  \begin{equation}\label{eq:after-rev-holder}
    \diff{R_t}{t} \ge
    \frac{\rho}{2\pi \sqrt{1-\rho^2}}
    \left(\E_\rho\exp
      \Big(\beta\frac{v_t^2 + w_t^2 - 2 \rho v_t w_t}{2(1-\rho^2)}\Big)
    \right)^{-1/\beta}
    \Big(\E_\rho |\grad v_t - \grad w_t|^{2r}\Big)^{1/r}
  \end{equation}
  with $\beta$ and $r$ yet to be determined.
  Let us first consider the exponential term in~\eqref{eq:after-rev-holder}.
  Since $2|v_t w_t| \le v_t^2 + w_t^2$, we have
  \begin{align}
    \E_\rho\exp
      \Big(\beta\frac{v_t^2 + w_t^2 - 2 \rho v_t w_t}{2(1-\rho^2)}\Big)
    &\le
    \E_\rho\exp
      \Big(\beta\frac{v_t^2 + w_t^2}{2(1-\rho)}\Big) \notag \\
    &\le \left(
      \E\exp\Big(\beta\frac{v_t^2}{1-\rho}\Big)
      \E\exp\Big(\beta\frac{w_t^2}{1-\rho}\Big)
    \right)^{1/2}, \label{eq:bounding-exp-squared}
  \end{align}
  where we used the Cauchy-Schwarz inequality in the last line.
  Recall from Lemma~\ref{lem:grad-bound-v} that $v_t$ and $w_t$
  are both $k_t$-Lipschitz. Thus, we can
  apply Lemma~\ref{lem:lipschitz-gaussian} with $f = v_t / k_t$
  and $\lambda = 2\beta k_t^2 / (1-\rho)$; we see that
  if $\lambda = 2\beta k_t^2 / (1-\rho) \le \frac{1}{2}$, then
  \[
    \E\exp\Big(\beta\frac{v_t^2}{1-\rho}\Big) \le
    C e^{\lambda M_t^2},
  \]
  where $M_t$ is a median of $v_t$.
  Applying the same argument to $w_t$ and plugging the result
  into~\eqref{eq:bounding-exp-squared}, we have
  \[
    \E_\rho\exp
      \Big(\beta\frac{v_t^2 + w_t^2 - 2 \rho v_t w_t}{2(1-\rho^2)}\Big)
      \le C e^{\lambda(M_t^2 + N_t^2)},
  \]
  where $N_t$ is a median of $w_t$.
  Going back to~\eqref{eq:after-rev-holder}, we have
  \begin{equation}\label{eq:prop-holder-with-med}
  \diff{R_t}{t} \ge
    \frac{c \rho}{\sqrt{1-\rho^2}}
    e^{-\frac{\lambda}{\beta}(M_t^2 + N_t^2)}
    \Big(\E_\rho |\grad v_t - \grad w_t|^{2r}\Big)^{1/r},
  \end{equation}
  with (recall) $\lambda = 2\beta k_t^2 / (1 - \rho) \le \frac{1}{2}$;
  hence, $\beta \le \frac{1}{4}(1-\rho) / k_t^2$.
  Recalling that $\frac{1}{r} - \frac{1}{\beta} = 1$, we see
  that~\eqref{eq:prop-holder-with-med} holds for any
  $r < \frac{1}{1 + 4 k_t^2/(1-\rho)}$.
  Finally, we invoke Lemma~\ref{lem:mean-median}
  to show that
  \[
    \exp\Big(-\frac{\lambda}{\beta} M_t^2\Big)
    =\exp\Big(-\frac{2 k_t^2 M_t^2}{1-\rho}\Big)
    \ge (c \E f (1 - \E f))^{2\frac{k_t^2 (1+k_t)^2}{1 - \rho}}
  \]
  (and similarly for $g$ and $N_t$). Plugging this
  into~\eqref{eq:prop-holder-with-med} completes the proof.
\end{proof}

\subsection{Proof of Proposition~\ref{prop:2-funcs-large-t}}

We are now prepared to prove Proposition~\ref{prop:2-funcs-large-t}
by combining Proposition~\ref{prop:2-funcs-holder} with
Lemmas~\ref{lem:grad-bound-v}
and~\ref{lem:2-funcs-poincare}.
Besides combining these three results, there is a small
technical obstacle: we know only that the integral
of $\diff{R_t}{t}$ is small; we don't know anything about
$\diff{R_t}{t}$ at specific values of $t$. So instead of
showing that $v_t$ is close to linear for every $t$, we will show
that for every $t$, there is a nearby $t^*$ such that
$v_{t^*}$ is close to linear. By ensuring that $t^*$ is close to $t$,
we will then be able to argue that $v_t$ is also close to linear.

\begin{proof}[Proof of Proposition~\ref{prop:2-funcs-large-t}]
  For any $0 < r < 1$, Lemma~\ref{lem:grad-bound-v} implies that
  \[
    \big(\E_\rho |\grad v_t - \grad w_t|^{2r}\big)^{1/r}
    \ge \frac{\big(\E |\grad v_t - \grad w_t|^{2}\big)^{1/r}}
    {2k_t^{2(1-r)/r}}.
  \]
  By Lemma~\ref{lem:2-funcs-poincare} applied to $v_t$ and $w_t$,
  if we set $a = \frac{1}{2}(\E \grad v_t + \E \grad w_t$) and we
  define $\epsilon(v_t) = \E (v_t(X) - \inr{X}{a} - \E v)^2$
  (and similarly for $\epsilon(w_t)$), then
  \[
    (\epsilon(v_t) + \epsilon(w_t))^{1/r} \le
    \frac{k_t^{2(1-r)/r}}{1-\rho} \big(\E_\rho |\grad v_t - \grad w_t|^{2r}\big)^{1/r}.
  \]
  Now we plug this into Proposition~\ref{prop:2-funcs-holder}
  to obtain
  \begin{equation}\label{eq:large-t-diff-R_t}
    (\epsilon(v_t) + \epsilon(w_t))^{1/r} \le
    C(t, \rho) m^{\frac{1-\rho}{2 k_t^2(1+k_t)^2}}
    \diff{R_t}{t}.
  \end{equation}

  Recall that $\delta(f, g) = \int_0^\infty \diff{R_s}{s}\ ds$.
  In particular,
  \[
    \alpha t \min_{t \le s \le t(1+\alpha)} \diffat{R_t}{t}{s}
    \le \int_{t}^{t(1+\alpha)} \diff{R_s}{s} \ ds \le \delta(f, g)
  \]
  and so there is some $s \in [t, t(1 + \alpha)]$ such that
  $\diffat{R_t}{t}{s} \le \frac{\delta}{\alpha t}$.
  If we apply this to~\eqref{eq:large-t-diff-R_t}
  with $t$ replaced by $s$ and with
  $r = \frac{1}{1 + 4 k_t^2/(1-\rho)}
  \le \frac{1}{1 + 4 k_s^2/(1-\rho)}$, we obtain
  \[
    \epsilon(v_s) + \epsilon(w_s)
    \le C(t, \rho)
    m^{r\frac{1-\rho}{2 k_t^2 (1 + k_t^2)}}
    \Big(\frac{\delta}{\alpha}\Big)^r.
  \]
  Since $\Phi$ is Lipschitz, if we denote
  $\E (f_s - \Phi(\inr{X}{a} - \E v_s))^2$ by $\epsilon(f_s)$
  (and similarly for $g_s$), then we have
  \[
    \epsilon(f_s) + \epsilon(g_s)
    \le C(t, \rho)
    m^{r\frac{1-\rho}{2 k_t^2(1 + k_t^2)}}
    \Big(\frac{\delta}{\alpha}\Big)^r.
  \]
  Note that $r = \frac{1-\rho}{1-\rho + 4 k_t^2}
  \ge \frac{1-\rho}{4(1+k_t^2)}$ and so
  \begin{equation}\label{eq:2-funcs-at-s}
    \epsilon(f_s) + \epsilon(g_s)
    \le C(t, \rho)
    m^{\frac{(1-\rho)^2}{8 k_t^2(1 + k_t^2)^2}}
    \Big(\frac{\delta}{\alpha}\Big)^r.
  \end{equation}

  Now we will need a lemma to show that $\epsilon(f_t)$ and
  $\epsilon(f_t)$ are small. We will prove the lemma after this proof
  is complete.

  \begin{lemma}\label{lem:partial-pullback}
    For any $t < s$ and any $h \in L_2(\R^n, \gamma_n)$,
    \[
      \E(P_t h)^2 \le \big(\E (P_s h)^2\big)^{t/s}
      \big(\E h^2\big)^{1 - t/s}.
    \]
  \end{lemma}

  To complete the proof of Proposition~\ref{prop:2-funcs-large-t},
  apply Lemma~\ref{lem:partial-pullback} with
  $h = f - P_s^{-1} \Phi(\inr{X}{a} - \E v_s)$ (note that
  $P_s^{-1} \Phi(\inr{X}{a} - \E v_s)$ exists
  by Lemma~\ref{lem:pull-back}, because $|a| \le k_s$).
  Since $\E h^2 \le \sup |h| \le 1$ and $s \le (1+\alpha) t$, we see that
  \[\epsilon(f_t) = \E (P_t h)^2 \le \big(\E (P_s h)^2\big)^{t/s}
    = \epsilon(f_s)^{1/(1+\alpha)}.
  \]
  Applying this (and the equivalent inequality for $g$)
  to~\eqref{eq:2-funcs-at-s}, we have
  \[
    \epsilon(f_t) + \epsilon(g_t)
    \le C(t, \rho)^{\frac{1}{1+\alpha}}
    m^{\frac{(1-\rho)^2}{8k_t^2(1 + k_t^2)^2(1+\alpha)}}
    \Big(\frac{\delta}{\alpha}\Big)^\frac{r}{1+\alpha},
  \]
  where $\epsilon(f_t)$ means
  $\E (f_t - P_{s-t}^{-1} \Phi(\inr{X}{a} - \E v_s))^2$
  and similarly for $\epsilon(g_t)$.
  Since $\alpha < 1$, $\frac{1}{2} \le \frac{1}{1+\alpha} \le 1$
  and so we can absorb the power $\frac{1}{1+\alpha}$ into the constant
  $C(t, \rho)$.
\end{proof}

\begin{proof}[Proof of Lemma~\ref{lem:partial-pullback}]
  Expand $P_s h$ in the Hermite basis as $P_s h = \sum b_\alpha H_\alpha$.
  Then
  \begin{align*}
    \E (P_s h)^2 &= \sum b_\alpha^2 \\
    \E (P_t h)^2 &= \sum b_\alpha^2 e^{2(s-t)|\alpha|} \\
    \E h^2 &= \sum b_\alpha^2 e^{2s|\alpha|}.
  \end{align*}
  By H\"older's inequality applied with the exponents
  $s/t$ and $s/(s-t)$,
  \begin{align*}
    \E (P_t h)^2
    &= \sum b_\alpha^{(s-t)/s} e^{2(s-t)|\alpha|} b_\alpha^{t/s} \\
    &\le \Big(\sum b_\alpha^2 e^{2s|\alpha|}\Big)^{(s-t)/s}
    \Big(\sum b_\alpha^2\Big)^{t/s} \\
    &= \big(\E h^2\big)^{(s-t)/s} \big(\E (P_s h)^2\big)^{t/s}.
    \qedhere
  \end{align*}
\end{proof}

\section{Robustness: time-reversal}\label{sec:time-rev}

The final step in proving Theorem~\ref{thm:2-funcs-robustness} is
to show that the conclusion of Proposition~\ref{prop:2-funcs-large-t}
implies that $f$ and $g$ are close to one of the equality cases.
In~\cite{MosselNeeman:12}, the authors used a spectral argument. However,
that spectral argument was responsible for the logarithmically
slow rates (in $\delta$)
that~\cite{MosselNeeman:12} showed. Here, we use a better
time-reversal argument that gives polynomial rates.
The argument here will need the function $f$ to take values only
in $\{0, 1\}$. Thus, we will first establish
Theorem~\ref{thm:2-funcs-robustness} for sets; having done so, it is not difficult
to extend it to functions using the equivalence, described
in Section~\ref{sec:results-robust}, between the set
and functional forms of Borell's theorem.

The main goal of a time-reversal argument is to bound $\E |h|$ from
above in terms of $\E |P_t h|$, for some function $h$. The difficulty is
that such bounds are not possible for general $h$. An illuminating example
is the function $h: \R \to \R$ given by $h(x) = \sgn(\sin(kx))$:
on the one hand, $\E |h| = 1$; on the other, $\E |P_t h|$ can be made
arbitrarily small by taking $k$ large.

The example above is problematic because there is a lot of cancellation
in $P_t h$. The essence of this section is that for the functions
$h$ we are interested in, there is a geometric reason which disallows
too much cancellation. Indeed, we are interested in functions $h$
of the form $1_A - 1_B$ where $B$ is a half-space. The negative part
of such a function is supported on $B$, while the positive part
is supported on $B^c$. As we will see, this fact allows us to bound
the amount of cancellation that occurs, and thus obtain a time-reversal
result:

\begin{proposition}\label{prop:time-rev}
  Let $B$ be a half-space and $A$ be any other set. There
  is an absolute constant $C$ such that
  for any $t > 0$,
  \[
    \gamma(A \symdiff B)
    \le C\max\Big\{
      \E |P_t 1_A - P_t 1_B|,
      (e^{2t} - 1)^{1/4} \sqrt{\E |P_t 1_A - P_t 1_B|}
    \Big\},
  \]
\end{proposition}

The main idea in Proposition~\ref{prop:time-rev} is in the following
lemma, which states that
if a non-negative function is supported on a half-space then $P_t$
will push strictly less than half of its mass onto the
complementary half-space.

\begin{lemma}\label{lem:time-rev}
  There is a constant $c > 0$ such that
  for any $b \in \R$,
  if $f: \R^n \to [0, 1]$ is supported on
  $\{x_1 \le b\}$ then for any $t > 0$,
  \[
    \E (P_t f) 1_{\{X_1 \ge e^{-t} b\}}
    \le \max\Big\{ \frac{1}{2} \E f - c \frac{(\E f)^2}{\sqrt{e^{2t} - 1}},
      \frac{3}{8} \E f
    \Big\}.
  \]
\end{lemma}

\begin{proof}
  Because $P_t$ is self-adjoint,
  \[
    \E (P_t f) 1_{\{X_1 \ge e^{-t} b\}}
    = \E f P_t 1_{\{X_1 \ge e^{-t} b\}}
    = \E f \Phi\left(\frac{X_1-b}{\sqrt{e^{2t} - 1}}\right).
  \]
  Now, the set $\{b - \E f \le x_1 \le b\}$ has measure
  at most $\phi(0) \E f$. In particular,
  $\E f 1_{\{b - \E f \le x_1 \le b\}} \le \phi(0) \E f \le \frac{1}{2} \E f$.

  Let $A = \{x_1 \le b - \E f\}$ and $B = \{b - \E f \le x_1 \le b\}$
  and recall that $f$ is supported on $\{x_1 \le b\}$, so that
  $f = f (1_A + 1_B)$.
  Now, 
  \[
    \Phi\left(\frac{x_1-b}{\sqrt{e^{2t} - 1}}\right) \le
    \begin{cases}
      \Phi\big(-\frac{\E f}{\sqrt{e^{2t} -1}}\big) & x \in A \\
      \frac{1}{2} & x \in B
    \end{cases}
  \]
  and so
  \begin{align}
    \E f \Phi\left(\frac{X_1-b}{\sqrt{e^{2t} - 1}}\right)
    &= \E 1_A f \Phi\left(\frac{X_1-b}{\sqrt{e^{2t} - 1}}\right)
    + \E 1_B f \Phi\left(\frac{X_1-b}{\sqrt{e^{2t} - 1}}\right) \notag \\
    &\le \Phi\left(-\frac{\E f}{\sqrt{e^{2t} - 1}}\right) \E 1_A f
    + \frac{1}{2} \E 1_B f  \notag\\
    &= \frac{1}{2} \E f - \left(\frac{1}{2} - \Phi\Big(-\frac{\E f}{\sqrt{e^{2t} - 1}}\Big)\right)
    \E f 1_A.
    \label{eq:lem-time-rev-1}
  \end{align}
  There is a constant $c > 0$ such that for all $x \ge 0$,
  $\Phi(-x) \le \max\{\frac{1}{2} - c x, \frac{1}{4}\}$.
  Applying this with $x = \frac{\E f}{\sqrt{e^{2t} - 1}}$, we have
  \[
    \eqref{eq:lem-time-rev-1}
    \le \frac{1}{2} \E f - \E f 1_A \min\Big\{
      c \frac{\E f}{\sqrt{e^{2t} - 1}}, \frac{1}{4}
    \Big\}
    \le \max\Big\{ \frac{1}{2} \E f - c \frac{(\E f)^2}{\sqrt{e^{2t} - 1}},
      \frac{3}{8} \E f
    \Big\}
  \]
  where in the last inequality, we recalled that $\E f 1_A \ge \frac{1}{2}
  \E f$.
\end{proof}

\begin{proof}[Proof of Proposition~\ref{prop:time-rev}]
  Without loss of generality, $B$ is the half-space $\{x_1 \le b\}$.
  Let $f$ be the positive part of $1_A - 1_B$ and let $g$
  be the negative part, so that $\gamma(A \symdiff B) = \E f + \E g$.
  Note that $f$ is supported on $B^c$ and $g$ is supported on $B$.

  Without loss of generality, $\E f \ge \E g$;
  Lemma~\ref{lem:time-rev} implies that if
  $\E f \le C \sqrt{e^{2t} - 1}$ then
  \begin{equation}\label{eq:time-rev-1}
    2 \E (1_B P_t f + 1_{B^c} P_t g)
    \le \E f + \E g - c \frac{(\E f + \E g)^2}{\sqrt{e^{2t} - 1}}.
  \end{equation}
  On the other hand, if
  $\E f \ge C \sqrt{e^{2t} - 1}$ then
  \begin{equation}\label{eq:time-rev-1a}
    2 \E (1_B P_t f + 1_{B^c} P_t g)
    \le \frac{3}{4} \E f + \E g \le \frac{7}{8} (\E f + \E g).
  \end{equation}

  Thus,
  \begin{align*}
    \E |P_t f - P_t g|
    &= \E P_t f + \E P_t g - 2 \E \min\{P_t f, P_t g\} \\
    &= \E f + \E g - 2 \E \min\{P_t f, P_t g\} \\
    &\ge \E f + \E g - 2 \E (1_B P_t f + 1_{B^c} P_t g) \\
    &\ge \min\Big\{c \frac{(\E f + \E g)^2}{\sqrt{e^{2t} - 1}},
                   \frac{\E f + \E g}{8}\Big\},
  \end{align*}
  Where we have applied~\eqref{eq:time-rev-1} and~\eqref{eq:time-rev-1a}
  in the last inequality. Now there are two cases, depending on which
  term in the minimum is smaller: if the first term is smaller then
  \[
   \E f + \E g
   \le C (e^{2t} - 1)^{1/4} \sqrt{\E |P_t f - P_t g|};
  \]
  otherwise, the second term in the minimum is smaller and
  \[
   \E f + \E g \le 8 \E |P_t f - P_t g|.
  \]
  In either case,
  \[
    \gamma(A \symdiff B) = \E f + \E g
    \le C\max\Big\{
      \E |P_t f - P_t g|,
      (e^{2t} - 1)^{1/4} \sqrt{\E |P_t f - P_t g|}
    \Big\},
  \]
  as claimed.
\end{proof}

\subsection{Synchronizing the time-reversal}

Proposition~\ref{prop:time-rev} would be enough if we knew that
$\E (P_t 1_A - P_t 1_B)^2$ were small. Now,
Proposition~\ref{prop:2-funcs-large-t} and
Lemma~\ref{lem:pull-back} imply
that $\E (P_t 1_A - P_{t+s} 1_B)^2$ is small, for some $s \ge 0$.
In this section, we will show that if $e^{-t} = \rho$ then $s$ must
be small. Now, this is not necessarily the case for arbitrary sets
$A$; in fact, for any $s > 0$ one can find $A$ such that
$\E (P_t 1_A - P_{t+s} 1_B)^2$ is arbitrarily small. Fortunately,
we have some extra information on $A$: we know that it is
almost optimally noise stable with parameter $\rho$. In particular,
if $e^{-t} = \rho$ then $\E 1_A P_t 1_A$ is close to
$\E 1_B P_t 1_B$.

Using this extra information, the proof of robustness
proceeds as follows:
since $\E 1_A P_t 1_A$ is close to $\E 1_B P_t 1_B$ and $P_t 1_A$
is close to $P_{t+s} 1_B$, we will show that
$\E 1_B P_{t+s} 1_B$ is close to $\E 1_B P_t 1_B$. But we know
all about $B$: it is a half-space. Therefore, we can find
explicit and accurate estimates for $\E 1_B P_{t+s} 1_B$
and $\E 1_B P_t 1_B$ in terms of $t$, $s$ and $\gamma_n(B)$;
using them, we can conclude that $s$ is
small. Now, if $s$ is small then
we can show (again, using explicit estimates)
that $\E (P_t 1_B - P_{t+s} 1_B)^2$ is small. Since
$\E (P_t 1_A - P_{t+s} 1_B)^2$ is small (this was our starting point,
remember),
we can apply the triangle inequality to conclude
that $\E (P_t 1_A - P_t 1_B)^2$ is small.
Finally, we can apply Proposition~\ref{prop:time-rev}
to show that $\E |1_A - 1_B|$ is small.

\begin{proposition}\label{prop:small-s}
  For every $t$, there is a $C(t)$ such that the following holds.
  For sets $A, A' \subset \R^n$, suppose that $B, B' \subset \R^n$
  are parallel half-spaces with
  $\gamma(A) = \gamma(B)$, $\gamma(A') = \gamma(B')$.
  If there exist $s, \epsilon_1, \epsilon_2 > 0$ such that
  \[
    \E (P_t 1_A - P_{t+s} 1_B)^2 \le \epsilon_1^2
  \]
  and
  \[
    \E 1_A P_t 1_{A'} \ge \E 1_B P_t 1_{B'} - \epsilon_2
  \]
  then
  \[
    \big(\E (P_t 1_A - P_t 1_B)^2\big)^{1/2}
    \le C(t) \frac{\epsilon_1 + \epsilon_2}
    {\big(I(\gamma(A)) I(\gamma(A'))\big)^{C(t)}},
  \]
  where $I(x) = \phi(\Phi^{-1}(x))$.
\end{proposition}

Rather than prove Proposition~\ref{prop:small-s} all at once,
we have split the part relating
$\E (P_t 1_B - P_{t+s} 1_B)^2$
and $\E 1_B (P_t 1_{B'} - P_{t+s} 1_{B'})$ into a separate lemma.

\begin{lemma}\label{lem:norm-noise-sens}
  For every $t$
  there is a $C(t)$ such that
  for any parallel half-spaces $B$ and $B'$, and for every $s > 0$,
  \[
    \big(\E (P_t 1_B - P_{t+s} 1_B)^2\big)^{1/2}
    \le C(t) \frac{\E 1_B (P_t 1_{B'} - P_{t+s} 1_{B'})}
    {\big(I(\gamma(B)) I(\gamma(B'))\big)^{C(t)}}.
  \]
\end{lemma}

\begin{proof}
  First of all, one can easily check through integration by parts
  that for a smooth function $f: \R \to \R$,
  \begin{equation}\label{eq:int-by-parts}
    \int_b^\infty \phi(x) (Lf)(x)\, dx = -f'(b)\phi(b).
  \end{equation}
  By rotating $B$ and $B'$, we can assume that $B = \{x_1 \le a\}$
  and $B' = \{x_1 \le b\}$.
  Let $F_{ab}(t) = \E 1_B P_t 1_{B'}
  = \int_a^\infty \phi(x)
  \Phi\big(\frac{e^{-t} x - b}{\sqrt{1-e^{-2t}}}\big)\, dx$
  and consider its derivative: by~\eqref{eq:int-by-parts},
  \begin{align*}
    F_{ab}'(t)
    &= \int_a^\infty \phi(x)
    L \Phi\Big(\frac{e^{-t} x - b}{\sqrt{1-e^{-2t}}}\Big)\, dx \\
    &= -k_t \phi(a)\phi\Big(\frac{e^{-t} a - b}{\sqrt{1-e^{-2t}}}\Big) \\
    &= -\frac{k_t}{2\pi}
    \exp\Big(-\frac{a^2 + b^2 - 2 e^{-t} ab}{2(1 - e^{-2t})} \Big) \\
    &\le -\frac{k_t}{2\pi}
    \exp\Big(-\frac{a^2 + b^2}{1 - e^{-2t}} \Big).
  \end{align*}
  Now, $k_t$ is decreasing in $t$ and $\exp(-x/(1-e^{-2t}))$ is increasing
  in $t$. In particular, for any $\tau \in [t, t+s]$,
  \[
    F'_{ab}(\tau) \le -\frac{k_{t+s}}{2\pi}
    \exp\Big(-\frac{a^2 + b^2}{1 - e^{-2t}} \Big).
  \]
  Hence,
  \begin{equation}\label{eq:noise-sens-s-small}
    F_{ab}(t) - F_{ab}(t+s) \ge -s \max_{s \le \tau \le t} F_{ab}'(\tau)
    \ge
    \frac{s k_{t+s}}{2\pi}
    \exp\Big(-\frac{a^2 + b^2}{1 - e^{-2t}} \Big).
  \end{equation}
  If $s$ is large, this is a poor bound because
  $s k_{t+s}$ decreases exponentially in $s$. However, when $s \ge 1$
  we can instead use
  \begin{equation}\label{eq:noise-sens-s-large}
    F_{ab}(t) - F_{ab}(t+s) \ge F_{ab}(t) - F_{ab}(t+1)
    \ge \frac{k_{t+1}}{2\pi}
    \exp\Big(-\frac{a^2 + b^2}{1 - e^{-2t}} \Big).
  \end{equation}

  Equations~\eqref{eq:noise-sens-s-small}
  and~\eqref{eq:noise-sens-s-large} show that if
  $\E 1_B (P_t 1_{B'} - P_{t+s} 1_{B'})$ is small then $s$ must be small.
  The next step, therefore, is to control
  $\E (P_t 1_B - P_{t+s} 1_B)^2$ in terms of $s$.
  Now,
  \begin{align}
    \E (P_t 1_B - P_{t+s} 1_B)^2
    &= \E \big((P_t 1_B)^2 + (P_{t+s} 1_B)^2 - 2(P_t 1_B) (P_{t+s} 1_B)
    \big) \notag \\
    &= \E 1_B \big(P_{2t} 1_B + P_{2(t+s)} 1_B - 2P_{2t+s} 1_B\big) \notag \\
    &= \big(F_{aa}(2t) - F_{aa}(2t+s)\big) - \big(F_{aa}(2t+s) - F_{aa}(2t+2s)\big) \notag \\
    &\le s\big(F_{aa}'(2t) - F_{aa}'(2t+2s)\big),
    \label{eq:bound-norm-by-F'}
  \end{align}
  where the inequality follows because
  \[
    F_{aa}'(t) = -\frac{k_t}{2\pi}
    \exp\Big(-\frac{(1 - e^{-t}) a^2}{1 - e^{-2t}}\Big)
    = -\frac{k_t}{2\pi}
    \exp\Big(-\frac{a^2}{1 + e^{-t}}\Big)
  \]
  and so $F'_{aa}$ is an increasing function.
  To control the right hand side
  of~\eqref{eq:bound-norm-by-F'}, we go to the second derivative of $F$:
  \[
    F''(t)
    = \frac{e^{2t}}{2\pi (e^{2t} - 1)^{3/2}}
    \exp\Big(-\frac{a^2}{1+e^{-t}}\Big)
    + \frac{1}{2\pi \sqrt{e^{2t} - 1}} \frac{a^2 e^{-t}}{(1 + e^{-t})^2}
    \exp\Big(-\frac{a^2}{1+e^{-t}}\Big)
  \]
  This is decreasing in $t$; hence
  \begin{equation}\label{eq:bound-norm-by-s}
    \E (P_t 1_B - P_{t+s} 1_B)^2
    \le s\big(F'(2t) - F'(2t+2s)\big)
    \le 2s^2 F''(2t).
  \end{equation}

  We will now complete the proof by combining our upper bound on
  $\E (P_t 1_B - P_{t+s} 1_B)^2$ with our lower bounds on
  $\E 1_B (P_t 1_{B'} - P_{t+s} 1_{B'})$.
  First, assume that $s \le 1$. Then $k_{t+s} \ge k_{t+1}$ and
  so~\eqref{eq:noise-sens-s-small} plus~\eqref{eq:bound-norm-by-s}
  implies that
  \begin{align*}
    \big(\E (P_t 1_B - P_{t+s} 1_B)^2\big)^{1/2}
    &\le 2\pi \exp\Big(\frac{a^2 + b^2}{1-e^{-2t}}\Big)
    \frac{\sqrt{2 F''(2t)}}{k_{t+1}}
    \E 1_B (P_t 1_{B'} - P_{t+s} 1_{B'}) \\
    &= 2\pi^{1-\frac{2}{1-e^{-2t}}}
    \frac{\sqrt{2 F''(2t)}}{k_{t+1}}
    \frac{\E 1_B (P_t 1_{B'} - P_{t+s} 1_{B'})}
    {\big(I(\gamma(B)) I(\gamma(B'))\big)^{\frac{2}{1-e^{-2t}}}}.
  \end{align*}
  If we take
  $C(t) \ge \max\{\sqrt{2 F''(2t)} / k_{t+1}, 2/(1-e^{-2t})\}$
  then the Lemma holds in this case.
  On the other hand, if $s > 1$ then~\eqref{eq:noise-sens-s-large}
  implies that
  \[
    \frac{2\pi^{1-\frac{2}{1-e^{-2t}}}}{k_{t+1}}
    \frac{\E 1_B (P_t 1_{B'} - P_{t+s} 1_{B'})}
    {\big(I(\gamma(B)) I(\gamma(B'))\big)^{\frac{2}{1-e^{-2t}}}}
    \ge 1.
  \]
  Since $\E (P_t 1_B - P_{t+s} 1_B)^2 \le 1$ trivially, the Lemma
  holds in this case provided that
  \[C(t) \ge \max\{ 1/k_{t+1}, 2/(1-e^{-2t})\}.\qedhere\]
\end{proof}

\begin{proof}[Proof of Proposition~\ref{prop:small-s}]
  By the Cauchy-Schwarz inequality,
  \[
    \E 1_A P_t 1_A
    \le \E 1_A P_{t+s} 1_B + \sqrt{\E (P_t 1_A - P_{t+s} 1_B)^2 }
    \le \E 1_A P_{t+s} 1_B + \epsilon_1.
  \]
  Moreover, $\E 1_A P_{t+s} 1_B \le \E 1_B P_{t+s} 1_B$ since
  $B$ is a super-level set of $P_{t+s} 1_B$ with the same volume as
  $A$. Thus,
  \begin{align*}
    \E 1_B P_t 1_B - \epsilon_2
    &\le \E 1_A P_t 1_A \\
    &\le \E 1_A P_{t+s} 1_B + \epsilon_1 \\
    &\le \E 1_B P_{t+s} 1_B + \epsilon_1.
  \end{align*}
  By Lemma~\ref{lem:norm-noise-sens},
  \[
    \big(\E (P_t 1_B - P_{t+s} 1_B)^2\big)^{1/2}
    \le C(t) \E 1_B (P_t 1_B - P_{t+s} 1_B)
    \le C(t) (\epsilon_1 + \epsilon_2)
  \]
  Finally, the triangle inequality gives
  \begin{align*}
    \big(\E (P_t 1_A - P_t 1_B)^2\big)^{1/2}
    &\le \big(\E (P_t 1_A - P_{t+s} 1_B)^2\big)^{1/2}
    + \big(\E (P_t 1_B - P_{t+s} 1_B)^2\big)^{1/2} \\
    &\le \epsilon_1 + C(t) (\epsilon_1 + \epsilon_2).
  \end{align*}
  Of course, 1 can be absorbed into the constant $C(t)$.
\end{proof}

\subsection{Proof of robustness}

\begin{proof}[Proof of Theorem~\ref{thm:2-funcs-robustness}]
  First, define $t$ by $e^{-t} = \rho$.
  We then have
  $k_t^2 = \frac{\rho^2}{1-\rho^2}$
  and so the exponent of $\delta$
  in Proposition~\ref{prop:2-funcs-large-t} becomes
  \begin{equation}\label{eq:exponent}
    \frac{1}{1 + 4 \frac{\rho^2}{(1-\rho^2)(1-\rho)}} \cdot \frac{1}{1+\alpha}
  = \frac{(1 - \rho^2)(1-\rho)}{1 - \rho + 3 \rho^2 + \rho^3}
  \cdot \frac{1}{1+\alpha}.
  \end{equation}
  Of course, we can define $\alpha > 0$ (depending on $\rho$)
  so that~\eqref{eq:exponent} is
  \[
    \eta := \frac{(1 - \rho^2)(1-\rho)}{1 + 3\rho}.
  \]

  Now suppose that $f = 1_A$ and $g = 1_{A'}$
  for some $A, A' \subset \R^n$.
  Proposition~\ref{prop:2-funcs-large-t} implies
  that there are $a \in \R^n$ and $b \in \R$ such that
  $|a| \le k_t$ and
  \[
    \E \big((P_t 1_A)(X) - \Phi(\inr{a}{X} - b)\big)^2
    \le C(\rho) m^{c(\rho)}
    \delta^\eta.
  \]
  Since $|a| \le k_t$, Lemma~\ref{lem:pull-back} implies that
  we can find some $s > 0$ and a half-space $B$ such that
  $\Phi(\inr{a}{x} - b) = (P_{t+s} 1_B)(x)$; then
  \begin{equation}\label{eq:2-funcs-proof-half-space}
    \E (P_t 1_A - P_{t+s} 1_B)^2
    \le C(\rho) m^{c(\rho)}
    \delta^\eta.
  \end{equation}
  At this point, it isn't clear that $\gamma(A) = \gamma(B)$; however,
  we can ensure this by modifying $B$ slightly:
  \[
    \E (P_t 1_A - P_{t+s} 1_B)^2
    \ge (\E P_t 1_A - \E P_{t+s} 1_B)^2 = (\gamma(A) - \gamma(B))^2.
  \]
  Therefore let $\tilde B$ be a translation of $B$ so that
  $\gamma(\tilde B) = \gamma(A)$. By the triangle inequality,
  \begin{align*}
    \big(\E (P_t 1_A - P_{t+s} 1_{\tilde B})^2\big)^{1/2}
    &\le
    \big(\E (P_t 1_A - P_{t+s} 1_B)^2\big)^{1/2}
    + \big(\E (P_{t+s} 1_B - P_{t+s} 1_{\tilde B})^2\big)^{1/2} \\
    &\le
    \big(\E (P_t 1_A - P_{t+s} 1_B)^2\big)^{1/2}
    + |\gamma(B) - \gamma(\tilde B)|^{1/2} \\
    &\le
    2\big(\E (P_t 1_A - P_{t+s} 1_B)^2\big)^{1/2}.
  \end{align*}
  By replacing $B$ with $\tilde B$, we can assume
  in~\eqref{eq:2-funcs-proof-half-space} that $\gamma(A) = \gamma(B)$
  (at the cost of increasing $C(\rho)$ by a factor of 2).

  Now we apply Proposition~\ref{prop:small-s} with
  $\epsilon_1^2 = C(\rho) m^{c(\rho)}
  \delta^\eta$ and $\epsilon_2 = \delta$. The conclusion
  of Proposition~\ref{prop:small-s} leaves us with
  \begin{align*}
    \big(\E (P_t 1_A - P_t 1_B)^2\big)^{1/2}
    &\le C(\rho) m^{c(\rho)}
    (\epsilon_1 + \epsilon_2) \\
    &\le C(\rho) m^{c(\rho)}
    \delta^{\eta/2}.
  \end{align*}
  where we have absorbed the constant $C(t)$ from
  Proposition~\ref{prop:time-rev} into $C(\rho)$ and $c(\rho)$.
  Since $\E |X| \le (\E X^2)^{1/2}$ for any random variable $X$, we
  may apply Proposition~\ref{prop:time-rev}:
  \begin{align*}
    \gamma(A \symdiff B)
    &\le C(\rho) \sqrt{\E |P_t 1_A - P_t 1_B|} \\
    &\le C(\rho) \big(\E (P_t 1_A - P_t 1_B)^2\big)^{1/4} \\
    &\le C(\rho) m^{c(\rho)} \delta^{\eta/4}.
  \end{align*}
  By applying the same argument to $A'$ and $B'$,
  this establishes Theorem~\ref{thm:2-funcs-robustness} in the case
  that $f$ and $g$ are indicator functions.

  To extend the result to other functions, note that
  $\E J(f(X), g(Y)) = \E J(1_A(\tilde X), 1_{A'}(\tilde Y))$ where
  $\tilde X$ and $\tilde Y$ are $\rho$-correlated Gaussian vectors
  in $\R^{n+1}$, and
  \begin{align*}
    A &= \{(x, x_{n+1}) \in \R^{n+1}: x_{n+1} \ge \Phi^{-1}(f(x))\} \\
    A' &= \{(x, x_{n+1}) \in \R^{n+1}: x_{n+1} \ge \Phi^{-1}(g(x))\}.
  \end{align*}
  Moreover, $\E f = \gamma_{n+1}(A)$ and $\E g = \gamma_{n+1}(A')$.
  Applying Theorem~\ref{thm:2-funcs-robustness} for indicator functions
  in dimension $n+1$, we find a half-space $B$ so that
  \begin{equation}\label{eq:sets-to-functions}
    \gamma_{n+1} (A \symdiff B)
    \le C(\rho) m^{c(\rho)} \delta^{\eta/4}.
  \end{equation}
  By slightly perturbing $B$, we can assume that it does not take the
  form $\{x_i \ge b\}$ for any $1 \le i \le n$; in particular, this
  means that we can write $B$ in the form
  \[
    B = \{(x, x_{n+1}) \in \R^n: x_{n+1} \ge \inr{a}{x} - b\}.
  \]
  for some $a \in \R^n$ and $b \in \R$.
  But then
  \[
    \gamma_{n+1} (A \symdiff B)
    = \E |f(X) - \Phi(\inr{a}{X} - b)|;
  \]
  combined with~\eqref{eq:sets-to-functions}, this completes the proof.
\end{proof}

\section{Optimal dependence on $\rho$}\label{sec:rho}

In this section, we will prove Theorem~\ref{thm:1-func-robustness}. To do so
we need to improve the dependence on $\rho$ that appeared
in Theorem~\ref{thm:2-funcs-robustness}.
Before we begin, let us list the places where the dependence on $\rho$
can be improved:
\begin{enumerate}
  \item In Proposition~\ref{prop:2-funcs-holder},
    we needed to control
    \[
      \E_\rho \exp\Big(\beta \frac{v_t^2(X) + w_t^2(Y) - 2 \rho v_t(X) w_t(Y)}
      {2(1-\rho^2)}\Big).
    \]
    Of course, the denominator of the exponent blows up as $\rho \to 1$.
    However, if $v_t = w_t$ then the numerator goes to zero (in law, at least)
    at the same rate. In this case, therefore,
    we are able to bound the above expectation by
    an expression not depending on $\rho$.

  \item In the proof of Proposition~\ref{prop:2-funcs-large-t},
    we used an $L_\infty$ bound on $|\grad v_t|$ and $|\grad w_t|$ to show
    that for some $r < 1$,
    \[
      \E_\rho\big(|\grad v_t(X) - \grad w_t(Y)|^{2}\big)^{1/r}
      \le C(t) \E_\rho\big(|\grad v_t(X) - \grad w_t(Y)|^{2r}\big)^{1/r}.
    \]
    This inequality is not sharp in its $\rho$-dependence because
    when $v_t = w_t$,
    the left hand side shrinks like $(1-\rho)^{1/r}$ as $\rho \to 1$,
    while the right hand side shrinks like $1-\rho$.
    We can get the right $\rho$-dependence by using an $L_p$ bound on
    $|\grad v_t(X) - \grad v_t(Y)|$
    when applying H\"older's inequality, instead of an $L_\infty$ bound.

  \item In applying Proposition~\ref{prop:small-s}, we were forced to take
    $e^{-t} = \rho$. Since most of our bounds have a (necessary) dependence
    on $t$, this causes a dependence on $\rho$ which is not optimal.
    To get around this, we will use the subadditivity property
    of Kane~\cite{Kane:11}, Kindler and O'Donnell~\cite{KindlerOdonnell:12}
    to show that we can actually choose certain values of
    $t$ such that $e^{-t}$ is much smaller than $\rho$.
    In particular, we can take $t$ to be quite
    large even when $\rho$ is close to 1.
\end{enumerate}

Once we have incorporated the first two improvements, we will obtain a better
version of Proposition~\ref{prop:2-funcs-large-t}:

\begin{proposition}\label{prop:1-func-large-t}
For any $\alpha, t > 0$, there is a constant $C(t, \alpha)$
such that for any $f: \R^n \to [0, 1]$,
there exist $a \in \R^n, b \in \R$ with $|a| \le k_t$ such that
\[
 \E \big(f_t(X) - \Phi(\inr{X}{a} - b)\big)^2
 \le C(t, \alpha) m^{\frac{(1 + k_t)^2}{1 + 8 k_t^2} - \alpha}
 \Big( \frac{\delta}{\rho\sqrt{1-\rho}} \Big)^{\frac{1}{1+8k_t^2} - \alpha}.
\]
where $k_t = (e^{2t}-1)^{-1/2}$, $\delta(f) = \E_\rho J(f(X), f(Y)) - J(\E f, \E f)$, and
$m(f) = \E f(1-\E f)$.

Moreover, this statement holds with a $C(t, \alpha)$ which, for any fixed $\alpha$,
is decreasing in $t$.
\end{proposition}

Once we have incorporated the third improvement above, we will use the
arguments of Section~\ref{sec:time-rev} to prove Theorem~\ref{thm:1-func-robustness}.

\subsection{A better bound on the auxiliary term}
First, we will tackle item 1 above. Our improved bound leads to a
version of Proposition~\ref{prop:2-funcs-holder} with the correct dependence
on $\rho$.

\begin{proposition}\label{prop:1-func-holder}
  Let $k_t = (e^{2t} - 1)^{-1/2}$. There are constants $0 < c, C < \infty$
  such that
  for any $t > 0$,
  if $r \le \frac{1}{1 + 8 k_t^2}$ 
  then
  \[
    \diff{R_t}{t} \ge \frac{\rho}{\sqrt{1-\rho^2}}
    (c m(f))^{(1+k_t)^2}
    \Big(\E |\grad v_t(X) - \grad v_t(Y)|^{2r}\Big)^{1/r}
  \]
  where $m(f) = \E f (1 - \E f)$.
\end{proposition}

To obtain this improvement, we note that for
a Lipschitz function $v$, $(v(X) - v(Y)) / \sqrt{1-\rho}$
satisfies a Gaussian tail bound that does not depend on $\rho$:
\begin{lemma}\label{lem:1-func-tail-bound}
  If $v: \R^n \to \R$ is $L$-Lipschitz then
  \[
    \Pr_\rho\Big(v(X) - v(Y) \ge L s \sqrt{2(1-\rho)}\Big)
    \le 1 - \Phi(s).
  \]
  In particular,
  if $4 \beta L^2 < 1$ then
  \[
    \E_\rho \exp\Big(\beta \frac{(v(X) - v(Y))^2}{(1-\rho)} \Big)
    \le \frac{1}{\sqrt{1 - 4 \beta L^2}}.
  \]
\end{lemma}

\begin{proof}
  Let $Z_1 = \frac{X + Y}{2}$ and $Z_2 = \frac{X - Y}{2}$,
  so that $\E Z_1^2 = \frac{1 + \rho}{2}$ and
  $\E Z_2^2 = \frac{1 - \rho}{2}$.
  Now we condition on $Z_1$: the function $v(Z_1 + Z_2) - v(Z_1 - Z_2)$
  is $2L$-Lipschitz in $Z_2$ and has conditional median zero
  (because it is odd in $Z_2$); thus
  \[
    \Pr_\rho\Big(v(Z_1 + Z_2) - v(Z_1 - Z_2) \ge Ls \sqrt{2(1-\rho)} \Big| Z_1\Big)
    \le 1 - \Phi(s).
  \]
  Now integrate out $Z_1$ to prove the first claim.

  Proving the second claim from the first one is a standard calculation.
\end{proof}

Next, we use the estimate of Lemma~\ref{lem:1-func-tail-bound} to prove a bound on
\[
 \E_\rho \exp\Big(
 \beta \frac{v_t^2(X) + v_t^2(Y) - 2 \rho v_t(X) w_t(Y)}{2(1-\rho^2)}
 \Big)
\]
that is better than the one from~\eqref{eq:bounding-exp-squared}
which was used to derive Proposition~\ref{prop:2-funcs-holder}.

\begin{lemma}\label{lem:1-func-exp-squared}
  There is a constant $C$ such that for any $t > 0$,
  and for any $\beta > 0$ with $6\beta k_t^2 \le 1$,
  \[
    \E_\rho \exp\Big(
    \beta \frac{v_t^2(X) + v_t^2(Y) - 2 \rho v_t(X) v_t(Y)}{2(1-\rho^2)}
    \Big)
    \le C e^{M_t^2/2},
  \]
  where $M_t$ is a median of $v_t$.
\end{lemma}

\begin{proof}
 We begin with the Cauchy-Schwarz inequality:
 \begin{align}
  &\E_\rho \exp\Big(
  \beta \frac{v_t^2(X) + v_t^2(Y) - 2 \rho v_t(X) v_t(Y)}{2(1-\rho^2)}
  \Big) \notag \\
  &= \E_\rho \exp\Big(
  \beta \frac{(v_t(X) - v_t(Y))^2}{2(1-\rho^2)}\Big)
  \exp\Big(\beta\frac{v_t(X) v_t(Y)}{1+\rho}
  \Big) \notag \\
  &\le \bigg(\E_\rho \exp\Big(2\beta \frac{(v_t(X) - v_t(Y))^2}{2(1-\rho^2)}\Big)\bigg)^{1/2}
  \bigg(\exp\Big(2\beta\frac{v_t(X)^2}{1+\rho}\Big)\bigg)^{1/2}.
  \label{eq:1-func-exp-squared}
 \end{align}
 Now, recall from Lemma~\ref{lem:grad-bound-v} that $v_t$ is $k_t$-Lipschitz. In particular,
 Lemma~\ref{lem:1-func-tail-bound} implies that if $8\beta k_t^2 \le 1$ then the first term
 of~\eqref{eq:1-func-exp-squared} is at most $\sqrt 2$.
 Finally, Lemma~\ref{lem:lipschitz-gaussian} implies that the second term
 of~\eqref{eq:1-func-exp-squared} is bounded by $C e^{M_t^2/2}$.
\end{proof}

\begin{proof}[Proof of Proposition~\ref{prop:1-func-holder}]
 First, follow the proof of Proposition~\ref{prop:2-funcs-holder}
 up until~\eqref{eq:after-rev-holder}. At this point, we can apply
 Lemma~\ref{lem:1-func-exp-squared} to obtain
 \[
  \diff{R_t}{t} \ge c\frac{\rho}{\sqrt{1-\rho^2}}
  e^{M_t^2/2}
  \Big(\E_\rho | \grad v_t(X) - \grad v_t(Y)|^{2r}\Big)^{1/r},
 \]
 and we conclude by applying Lemma~\ref{lem:mean-median}, which implies that
 \[e^{M_t^2/2} \ge (c m(f))^{(1+k_t)^2}. \qedhere\]
\end{proof}

\subsection{Higher moments of $|\grad v_t(X) - \grad v_t(Y)|$}

Here, we will carry out the second step of the plan outlined
at the beginning of Section~\ref{sec:rho}. The main result is an upper
bound on arbitrary moments of $|\grad v_t(X) - \grad v_t(Y)|$.

\begin{proposition}\label{prop:higher-moments}
  There is a constant $C$ such that
  for any $t > 0$ and any $1 \le q < \infty$,
\[
  \big(\E_\rho |\grad v_t(X) - \grad v_t(Y)|^q\big)^{1/q}
  \le
   C k_t^2 \sqrt {q(1-\rho)}
    \Big( (1 + k_t) \sqrt{\log (1 / m(f))} + \sqrt q k_t\Big).
\]
\end{proposition}

If we fix $q$ and $t$, then the bound of Proposition~\ref{prop:higher-moments}
has the right dependence on $\rho$. In particular, we will use it instead of the
uniform bound $|\grad v_t| \le k_t$, which does not improve as $\rho \to 1$.

There are two main tools in the proof of Proposition~\ref{prop:higher-moments}.
The first is a moment bound on the Hessian of $v_t$, which was proved in~\cite{MosselNeeman:12}.
In what follows, $\|\cdot\|_F$ denotes the Frobenius norm of a matrix.

\begin{proposition}\label{prop:hessian-moments}
 Let $H v_t$ denote the Hessian matrix of $v_t$. There is a constant $C$ such that
 for all $t > 0$ and all $1 \le q < \infty$,
 \[
  \big(\E \|H v_t\|_F^q\big)^{1/q} \le C k_t^2 \Big(
  (1 + k_t) \sqrt{\log{\frac{1}{m(f)}}} + \sqrt q k_t
  \Big)
 \]
\end{proposition}

The other tool in the proof of Proposition~\ref{prop:higher-moments} is a result
of Pinelis~\cite{Pinelis:98}, which will allow us to relate moments
of $|\grad v_t(X) - \grad v_t(Y)|$ to moments of $\|H v_t\|_F$.


\begin{proposition}\label{prop:frob-norm-concentration}
  Let $h: \R^n \to \R^k$ be a $C^1$ function and let $D h$ be the
  $n \times k$ matrix of its partial derivatives. 
  If $Z_1$ and $Z_2$ are independent, standard Gaussian vectors in $\R^n$
  then
  \[
    \big(\E |h(Z_1) - h(Z_2)|^q \big)^{1/q}
    \le C \sqrt q \big( \E \|Dh\|_F^q\big)^{1/q}
  \]
  for every $1 \le q < \infty$, where $C$ is a universal constant.
\end{proposition}

\begin{proof}
 Define $f: \R^{2n} \to \R^k$ by $f(Z) = h(Z_1) - h(Z_2)$ where $Z = (Z_1, Z_2)$.
 Pinelis~\cite{Pinelis:98} showed that if $\Psi: \R^k \to \R$ is a convex function
 then for any function $f: \R^{2n} \to \R^k$ with $\E f = 0$,
 \[
  \E \Psi(f(Z)) \le \E
  \Psi\Big(\frac{\pi}{2} Df(Z) \cdot \tilde Z\Big),
 \]
 where $\tilde Z$ is an independent copy of $Z$.
 Applying this with $\Psi(x) = |x|^q$, and noting that $D f =
 (\begin{smallmatrix}1 & 0 \\ 0 & -1\end{smallmatrix}) \otimes D h$, we obtain
 \[
  \E |f(Z)|^q \le C^q \E |Dh(Z_1) \cdot Z_2|^q.
 \]
 Now, $\E |A Z_2|^q \le (C\sqrt q)^{q/2} \|A\|_F$ for any fixed matrix $A$;
 if we apply this fact conditionally on $Z_1$, then we obtain
 \[
  \E |f(Z)|^q \le (C\sqrt q)^q \E \|Dh\|_F^q. \qedhere
 \]
\end{proof}

\begin{proof}[Proof of Proposition~\ref{prop:higher-moments}]
  Let $Z, Z_1$ and $Z_2$ be independent standard Gaussians on $\R^n$;
  set $X = \sqrt \rho Z + \sqrt{1-\rho} Z_1$ and $Y = \sqrt \rho Z + \sqrt{1-\rho} Z_2$ so
  that $X$ and $Y$ are standard Gaussians with correlation $\rho$. Conditioned on $Z$,
  define the function
  \[
   h(x) = \grad v_t(\sqrt Z + \sqrt{1-\rho} x), 
  \]
  so that $h(Z_1) = \grad v_t(X)$ and $h(Z_2) = \grad v_t(Y)$. Note that
  \[
   (D h)(x) = \sqrt{1-\rho} (H v_t)(\sqrt \rho Z + \sqrt{1-\rho} x);
  \]
  thus Proposition~\ref{prop:frob-norm-concentration} (conditioned on $Z$)
  implies that
  \[
   \E \big(|\grad v_t(X) - \grad v_t(Y)|^q \mid Z\big)
   \le \big(C \sqrt{q(1-\rho)}\big)^q \E \big(\|H v_t(X)\|_F^q \mid Z\big).
  \]
  Integrating out $Z$ and raising both sides to the power $1/q$, we have
  \[
    \big(\E |\grad v_t(X) - \grad v_t(Y)|^q\big)^{1/p}
    \le C \sqrt{q(1-\rho)} \big( \E \|H v_t\|_F^q\big)^{1/q}.
  \]
  We conclude by applying Proposition~\ref{prop:hessian-moments} to the right
  hand side.
\end{proof}

With the first two steps of our outline complete, we are ready to prove
Proposition~\ref{prop:1-func-large-t}. This proof is much like
the proof of Proposition~\ref{prop:2-funcs-large-t}, except that it uses
Propositions~\ref{prop:1-func-holder} and~\ref{prop:higher-moments} in
the appropriate places.

\begin{proof}[Proof of Proposition~\ref{prop:1-func-large-t}]
For any non-negative random variable $Z$ and any $0 < \alpha < 2$, $0 < r < 1$,
H\"older's inequality applied
with $p = 2r/\gamma$ implies that
\[
 \E Z^2 = \E Z^{\gamma} Z^{2-\gamma}
 \le \big(\E Z^{2r}\big)^{\gamma/(2r)} \big(\E Z^{2r(2-\gamma)/(2r-\gamma)}\big)^{(2r-\gamma)/(2r)}.
\]
In particular, if we set $q = 2r(2-\gamma)/(2r-\gamma)$ then we obtain
\begin{equation}\label{eq:holder}
 \big(\E Z^{2r}\big)^{1/r}
 \ge \left(\frac{\E Z^2}{\big(\E Z^q\big)^{(2-\gamma)/q}}\right)^{2/\gamma}.
\end{equation}
Now, set $Z = |\grad v_t(X) - \grad v_t(Y)|$, $a = \E \grad v_t$
and $\epsilon(v_t) = \E (v_t(X) - \inr{X}{a} - \E v_t)^2$.
Lemma~\ref{lem:2-funcs-poincare}
and Proposition~\ref{prop:higher-moments} then
imply that the right-hand side of~\eqref{eq:holder} is at least
\begin{multline*}
\left(
 \frac{2(1-\rho)\epsilon(v_t)}
 {\Big(c k_t^2 \sqrt{q(1-\rho)}\big((1+k_t) \sqrt{\log (1/m(f))} + \sqrt q k_t\big)\Big)^{2-\gamma}}
 \right)^{2/\gamma} \\
 = c \sqrt{1-\rho} \left(
  \frac{\epsilon(v_t)}{\Big(k_t^2 \sqrt q \big((1+k_t) \sqrt{\log (1/m(f))} + \sqrt q k_t\big)\Big)^{2-\gamma}}
 \right)^{2/\gamma}
\end{multline*}
Now define $\eta = 8 k_t^2 / (1 + 8 k_t^2)$ and choose $r = 1 - \eta$ (so as to satisfy the
hypothesis of Proposition~\ref{prop:2-funcs-holder}).
If we then define $\gamma = 2r - \alpha \eta = 2 - (2 + \alpha) \eta$ for some $0 < \alpha < 1$,
we will find that $q = 2r \frac{2+\alpha}{\alpha} \le 6/\alpha$.
In particular, the last displayed quantity is at least
\[
 (1-\rho) (c \alpha)^{(2-\gamma)/\gamma} \frac{\epsilon(v_t)^{2/\gamma}}{\big((k_t^3 + 1)\sqrt{\log (1/m(f))}\big)^{(2-\gamma)/\gamma}}
\]
Since $(k_t^3 + 1)^{(2-\gamma)/\gamma}$ depends only on $t$, we can put this all together
(going back to~\eqref{eq:holder}) to obtain
\begin{align*}
 \big(\E |\grad v_t(X) - \grad v_t(Y)|^{2r}\big)^{1/r}
 &\ge c(t, \alpha) (1-\rho) \frac{\epsilon(v_t)^{2/\gamma}}{\log^{C(t)} (1/m(f))} \\
 &= c(t, \alpha) (1-\rho) \frac{\epsilon(v_t)^{\frac{1 + 8k_t^2}{1 - 4\alpha k_t^2}}}{\log^{C(t)} (1/m(f))}.
\end{align*}
Combined with Proposition~\ref{prop:1-func-holder}, this implies
\begin{align}
 \diff{R_t}{t}
 &\ge
 c(t) \rho \sqrt{1-\rho} \frac{m(f)^{(1+k_t)^2}}{\log^{C(t)} (1/m(f))} \epsilon(v_t)^{\frac{1 + 8k_t^2}{1 - 4\alpha k_t^2}}
 \notag \\
 &\ge
 c(t, \alpha) \rho \sqrt{1-\rho} m(f)^{(1+k_t)^2 + \alpha} \epsilon(v_t)^{\frac{1 + 8k_t^2}{1 - 4\alpha k_t^2}},
 \label{eq:1-func-large-t-diff-R_t}
\end{align}
where the last line follows because for every $\alpha > 0$ and every $C$, there is a $C'(\alpha)$ such that
for every $x \le \frac 14$,
$\log^C(1/x) \le C'(\alpha) x^{-\alpha}$.
Now, with~\eqref{eq:1-func-large-t-diff-R_t} as an analogue of~\eqref{eq:large-t-diff-R_t}, we complete
the proof by following that of Proposition~\ref{prop:1-func-large-t}. Let us reiterate
the main steps: recalling that
$\delta = \int_0^\infty \diff{R_s}{s}\ ds$, we see that
for any $\alpha, t > 0$, there is
some $s \in [t, t(1+\alpha)]$ so that $\diffat{R_t}{t}{s} \le \frac{\delta}{\alpha t}$.
By~\eqref{eq:1-func-large-t-diff-R_t} applied with $t=s$, we have
\[
 \epsilon(v_s) \le C(t, \alpha) m^{\frac{(1 + k_t)^2(1-4\alpha k_t^2)}{1 + 8 k_t^2} - \alpha}
 \Big( \frac{\delta}{\rho\sqrt{1-\rho}} \Big)^{\frac{1-4\alpha k_t^2}{1+8k_t^2}}.
\]
Now, note that $\Phi$ is a contraction, and so Lemma~\ref{lem:partial-pullback} implies
that
\begin{multline*}
 \E \big(f_t(X) - P_{s-t}^{-1} \Phi(\inr{X}{\E \grad v_s} - \E v_s)\big)^2 \\
 \le C(t, \alpha) m^{\frac{(1 + k_t)^2(1-4\alpha k_t^2)}{1 + 8 k_t^2} - \alpha}
 \Big( \frac{\delta}{\rho\sqrt{1-\rho}} \Big)^{\frac{1-4\alpha k_t^2}{1+8k_t^2} - \alpha}.
\end{multline*}
By changing $\alpha$ and adjusting $C(t, \alpha)$ accordingly, we can put this inequality
into the form that was claimed in the proposition.

Finally, recall that $|\E \grad v_s| \le k_s$ by Lemma~\ref{lem:grad-bound-v}, and so
$P_{s-t}^{-1} \Phi(\inr{X}{\E \grad v_s} - \E v_s)$ can be written in the form
$\Phi(\inr{X}{a} - b)$ for some $a \in \R^n$, $b \in \R$ with $|a| \le k_t$.
\end{proof}

\subsection{On the monitonicity of $\delta$ with respect to $\rho$}

The final step in the proof of Theorem~\ref{thm:1-func-robustness} is to
improve the application of Lemma~\ref{lem:norm-noise-sens}.
Assuming, for now, that $f$ is the indicator function of a set $A$, the hypothesis
of Theorem~\ref{thm:1-func-robustness} tells us if $e^{-t} = \rho$ then
$\E 1_A P_t 1_A$ is almost as large as possible; that is, it is almost as 
large as $\E 1_B P_t 1_B$ where $B$ is a half-space of probability $\P(A)$.
This assumption allows us to apply Lemma~\ref{lem:norm-noise-sens}, but only
with $t = \log(1/\rho)$. In particular, this means that we will need to use this
value of $t$ in Proposition~\ref{prop:1-func-large-t}, which implies a poor
dependence on $\rho$ in our final answer.

To avoid all these difficulties, we will follow Kane~\cite{Kane:11}
and Kindler and O'Donnell~\cite{KindlerOdonnell:12}
to show if $\E 1_A P_t 1_A$ is almost as large as possible for $t = \log(1/\rho)$, then
it is also large for certain values of $t$ that are larger.

\begin{proposition}\label{prop:change-rho}
Suppose $A \subset \R^n$ has $\P(A) = 1/2$.
 If $\theta = \cos (k \cos^{-1} \rho)$ for some $k \in \N$, and
 \[
 J(1/2, 1/2; \rho) - \E_\rho J(1_A(X), 1_A(Y); \rho) \le \delta
 \]
 then
 \[
 J(1/2, 1/2; \theta) - \E_\theta J(1_A(X), 1_A(Y); \theta) \le k \delta
 \]
\end{proposition}

\begin{proof}
 Let $Z_1$ and $Z_2$ be independent standard Gaussians on $\R^n$ and define
 $Z(\gamma) = Z_1 \cos \gamma + Z_2 \sin \gamma$. Note that for any $\gamma$ and any
 $j \in \N$, $Z((j+1)\gamma)$ and $Z(j\gamma)$ have correlation $\cos \gamma$.
 In particular, if
 $\gamma = \cos^{-1}(\rho)$, then the union bound implies that
 \begin{align}
  \P_\theta(X \in A, Y \not \in A)
  &= \Pr(Z(0) \in A, Z(k\gamma) \not \in A) \notag \\
  &\le \sum_{j=0}^{k-1} \Pr(Z(j\gamma) \in A, Z((j+1)\gamma) \not \in A) \notag \\
  &= k \Pr_\rho(X \in A, Y \not \in A).\label{eq:union-bound}
 \end{align}
 The remarkable thing about this inequality is that it becomes equality when
 $A$ is a half-space of measure $1/2$, because in this case,
 $\Pr_\rho(X \in A, Y \not \in A) = \frac{1}{2\pi} \cos^{-1}(\rho)$.
 
 Recall that $\E_\rho J(1_A(X), 1_A(Y); \rho) = \Pr_\rho(X \in A, Y \in A)$.
 Thus, the hypothesis of the proposition can be rewritten as
 \[
  \Big(\frac{1}{2} - \frac{1}{2\pi} \cos^{-1}(\rho)\Big)
  - \Big(\P(A) - \Pr_\rho(X \in A, Y \not \in A)\Big) \le \delta,
 \]
 which rearranges to read
 \[
  \Pr_\rho(X \in A, Y \not \in A) \le \delta + \frac{1}{2\pi} \cos^{-1} \rho.
 \]
 By~\eqref{eq:union-bound}, this implies that
 \[
  \P_\theta(X \in A, Y \not \in A) \le k\delta + \frac{1}{2\pi} \cos^{-1} \theta,
 \]
 which can then be rearranged to yield the conclusion of the proposition.
\end{proof}

Let us point out two deficiencies in Proposition~\ref{prop:change-rho}:
the requirement that $\P(A) = 1/2$ and that $k$ be an integer. The first of
these deficiencies is responsible for the assumption $\E f = \frac{1}{2}$
in Theorem~\ref{thm:1-func-robustness}, and the second one prevents us from obtaining a
better constant in the exponent of $\delta$. 
Both of these restrictions come from the subadditivity
condition~\eqref{eq:union-bound}, which only makes sense for an integer $k$, and
only achieves equality for a half-space of volume $\frac 12$.
But beyond the fact that our proof fails, we have no reason not to believe that
some version of Proposition~\ref{prop:change-rho} is true without these restrictions.
In particular, we make the following conjecture:

\begin{conjecture}\label{conj:monotonicity}
 There is a function $k(\rho, a)$ such that
 \begin{itemize}
 \item for any fixed $a \in (0, 1)$, $k(\rho, a) \sim \sqrt{1-\rho}$ as $\rho \to 1$;
 \item for any fixed $a \in (0, 1)$, $k(\rho, a) \sim \rho$ as $\rho \to 0$; and
  \item for any $a \in (0, 1)$ and any $A \subset \R^n$ the quantity
  \[
    \frac{J(a, a; \rho) - \E_\rho J(1_A(X), 1_A(Y); \rho)\big)}{k(\rho, a)}
  \]
  is increasing in $\rho$.

 \end{itemize}
\end{conjecture}

If this conjecture were true, it would tell us that sets which are almost optimal
for some $\rho$ are also almost optimal for smaller $\rho$, where the function $k(\rho, a)$
quantifies the almost optimality.

In any case, let us move on to the proof of Theorem~\ref{thm:1-func-robustness}.
If the conjecture is true, then the following proof will directly benefit from
the improvement.

\begin{proof}[Proof of Theorem~\ref{thm:1-func-robustness}]
We will prove the theorem when $f$ is the indicator function of a set $A$.
The extension to general $f$ follows from the same argument that was made in
the proof of Theorem~\ref{thm:2-funcs-robustness}.

Fix $\epsilon > 0$. If $\rho_0$ is close enough to $1$ then for
every $\rho_0 < \rho < 1$, there is a $k \in \N$ such that
$k \cos^{-1}(\rho) \in
[\frac{\pi}{2} - \epsilon, \frac{\pi}{2} - \frac{\epsilon}{2}]$.
In particular, this means that $\cos(k \cos^{-1}(\rho))
\in [c_1(\epsilon), c_2(\epsilon)]$,
where $c_1(\epsilon)$ and $c_2(\epsilon)$ converge to zero
as $\epsilon \to 0$.
Moreover, this $k$ must satisfy
\[
 k \le \frac{C(\epsilon)}{\cos^{-1}(\rho)} \le \frac{C(\epsilon)}{\sqrt{1-\rho}}.
\]
Now let $\theta = \cos(k \cos^{-1}(\rho))$. By Proposition~\ref{prop:change-rho}, $A$ 
satisfies
\[
 J(1/2, 1/2; \theta) - \E_\theta J(1_A(X), 1_A(Y); \theta) \le C(\epsilon) \frac{\delta}{\sqrt{1-\rho}}.
\]
Now we will apply Proposition~\ref{prop:1-func-large-t} with $\rho$ replaced
by $\theta$ and $t = \log(1/\theta)$. Since $\theta \le c_2(\epsilon)$,
it follows that $k_t = \theta/\sqrt{1-\theta^2} \le c_3(\epsilon)$
(where $c_3(\epsilon) \to 0$ with $\epsilon$). Thus, the conclusion
of Proposition~\ref{prop:1-func-large-t} gives us $a \in \R^n$, $b \in R$ such that
\begin{align}
 \E \big((P_t 1_A)(X) - \Phi(\inr{X}{a} - b)\big)^2
 &\le C \Big(\frac{\delta}{\theta \sqrt{(1-\theta)(1-\rho)}}\Big)^{1 - c_4(\epsilon)} \notag \\
 &\le C(\epsilon) \Big(\frac{\delta}{\sqrt{1-\rho}}\Big)^{1 - c_4(\epsilon)}.
\label{eq:1-func-robustness-large-t}
\end{align}

Now we apply the same time-reversal argument as in Theorem~\ref{thm:2-funcs-robustness}:
Lemma 3.1 implies that there is some $s > 0$ and a half-space $B$ such that
\[
 \E (P_t 1_A - P_{t+s} 1_B)^2 \le C(\epsilon) (\delta/\sqrt{1-\rho})^{1 - c_4(\epsilon)}
\]
and we can assume, at the cost of increasing $C(\epsilon)$, that $\P(B) = \P(A)$. Then
Proposition~\ref{prop:small-s} implies that
\[
 \E (P_t 1_A - P_t 1_B)^2 \le C(\epsilon) (\delta/\sqrt{1-\rho})^{1 - c_4(\epsilon)},
\]
and we apply Proposition~\ref{prop:time-rev} (recalling that
$t$ is bounded above and below by constants depending on $\epsilon$)
to conclude that
\[
 \P(A \symdiff B) \le C(\epsilon) (\delta/\sqrt{1-\rho})^{1/4 - c_4(\epsilon) / 4}.
\]
Recall that $c_4(\epsilon)$ is some quantity tending to zero with $\epsilon$.
Therefore, we can derive the claim of the theorem from the equation
above by modifying $C(\epsilon)$.
\end{proof}

Finally, we will prove Corollary~\ref{cor:1-func-robustness}.
\begin{proof}[Proof of Corollary~\ref{cor:1-func-robustness}]
 Since $x y \le J(x, y)$, the hypothesis of
 Corollary~\ref{cor:1-func-robustness}
 implies that
 \[
  \E J(f(X), f(Y)) \ge \frac{1}{4} + \frac{1}{2\pi} \arcsin(\rho) - \delta.
  \]
  Now, consider Theorem~\ref{thm:1-func-robustness} with $\epsilon = 1/8$.
  If $\rho > \rho_0$ then apply it; if not, apply
  Theorem~\ref{thm:2-funcs-robustness}. In either case, the conclusion
  is that there is some $a \in \R^n$ such that
  \[
  \E |f(X) - \Phi(\inr{X}{a})| \le C(\rho) \delta^c.
  \]
  Setting $g(X) = \Phi(\inr{X}{a})$, H\"older's inequality implies that
  \begin{align*}
   \big|\E g(X) g(Y) - \E f(X) f(Y)\big|
   &= \big|\E (g(X) - f(X)) g(Y) + \E f(X)(g(Y) - f(Y))\big| \\
   &\le 2 \E |f - g|.
  \end{align*}
  In particular,
  \begin{equation}\label{eq:g-almost-optimal}
   \E g(X) g(Y) \ge \frac{1}{4} + \frac{1}{2\pi} \arcsin(\rho)
   - \delta - C(\rho) \delta^c.
  \end{equation}
  But the left hand side can be computed exactly:
  if $|a| = (e^{2t} - 1)^{-1/2}$
  and $A = \{x \in \R^n: x_1 \le 0\}$
  then
  \begin{align*}
   \E g(X) g(Y)
   &= \E P_t 1_A(X) P_t 1_A(Y) \\
   &= \E 1_A(X) P_{2t-\log(\rho)} 1_A(X) \\
   &= \frac{1}{4} + \frac{1}{2\pi} \arcsin(e^{-2t} \rho) \\
   &\le \frac{1}{4} + \frac{1}{2\pi} \arcsin(\rho)
   - \frac{1}{2\pi} \rho (1-e^{-2t}),
  \end{align*}
  where the last line used the fact that the derivative of $\arcsin$
  is at least 1. Combining this with~\eqref{eq:g-almost-optimal}, we have
  \begin{equation}\label{eq:small-t}
   1 - e^{-2t} \le C(\rho) \delta^c
  \end{equation}
  On the other hand,
  \[
   \E |g - 1_A| = 2(1/2 - \E g 1_A) = \frac{1}{2} - \frac{1}{\pi} \arcsin(e^{-t}) \le \sqrt{1 - e^{-2t}},
  \]
  which combines with~\eqref{eq:small-t} to prove that
  $\E |g - 1_A| \le C(\rho) \delta^c$. Applying the triangle inequality,
  we conclude that
  \[
   \E |f - 1_A| \le \E |f - g| + \E |g - 1_A| \le C (\rho) \delta^c.
   \qedhere
  \]
\end{proof}

\section{The robust ``majority is stablest'' theorem}

In this section, we prove Theorem~\ref{thm:robust-mist}.
For the rest of this section, we set
$\xi$ and $\sigma$ to be uniformly random elements in $\{-1, 1\}^n$ satisfying $\E \xi_i \sigma_i = \rho$
for all $i$.

We begin the proof of Theorem~\ref{thm:robust-mist} by recalling some Fourier-theoretic properties of
$\{-1, 1\}^n$.
For $S \subset [n]$, define $\chi_S: \{-1, 1\}^n \to \{-1, 1\}$ by $\chi_S(x) = \prod_{i \in S} x_i$.
Then $\{\chi_S: S \subset [n]\}$ form an orthonormal basis of $L_2(\{-1, 1\}^n)$. We will write $\hat f_S$
for the coefficients of $f$ in this basis; that is,
\begin{equation}\label{eq:fourier}
 f(x) = \sum_{S \subset [n]} \hat f_S \chi_S(x).
\end{equation}
Recall the Bonami-Beckner semigroup $Q_t$ defined by
\[
 (Q_t f)(\xi) = \E_{e^{-t}} (f(\sigma) \mid \xi),
\]
and denote $Q_t f$ by $f_t$; then
\[
 \Stab_\rho(f) = \E_\rho f(\xi) f(\sigma) = \E f f_{\log(1/\rho)}.
\]

\subsection{The invariance principle}

Note that any function $f: \{-1, 1\}^n \to \R$ can be extended to a multilinear function
from on $\R^n$ through the Fourier expansion~\eqref{eq:fourier}: since $\chi_S(x)$ is defined for all $x \in \R^n$,
we may define $g(x)$ for $x \in \R^n$ by $g(x) = \sum_S \hat f_S \chi_S(x)$.
We will say that $g$ is the \emph{multilinear extension of $f$}; note that
$g$ and $f$ agree on $\{-1, 1\}^n$, thereby justifying the term ``extension.''

Let us remark on some well-known and important properties of multlinear polynomials.
First of all, since $\E \xi_i = \E X_i = 0$ and $\E \xi_i^2 = \E X_i^2 = 1$, it
is trivial to check that for multlinear functions $f$ and $g$,
\begin{align*}
 \E f(\xi) &= \E f(X) \\
 \E f^2(\xi) &= \E f^2(X) \\
 \E_\rho f(\xi) g(\sigma) &= \E f(X) g(Y).
\end{align*}
It is also easy to check that if $f$ is a multilinear polynomial then for any $t > 0$,
$Q_t f$ and $P_t f$ are the same polynomial. In particular, there is no ambiguity
in using the notation $f_t$ for both $P_t f$ and $Q_t f$.

Despite these similarities,
$g(X)$ and $g(\xi)$ can have very different distributions in
general (for example, if $g(x) = x_1$).
The main technical result of~\cite{MoOdOl:10}
is that when $f$ has low influence and $t > 0$, then $f_t(X)$ and
$f_t(\xi)$ have similar distributions.
We will quote a much less general statement then
the one proved in~\cite{MoOdOl:10}, which will nevertheless be sufficient for
our purposes. In particular, we will only need to know that if
$g(\xi)$ takes values in $[0, 1]$, then $g(X)$ mostly takes values in
$[0, 1]$.
Before stating the theorem from~\cite{MoOdOl:10},
let us introduce some notation: for a function $f$ taking values in $\R$,
let $\bar f$ be its truncation which takes values in $[0, 1]$:
\[
 \bar f(x) = 
   \begin{cases}
    0 & \text{ if $f(x) < 0$ } \\
    f(x) & \text{ if $0 \le f(x) \le 1$ } \\
    1 & \text{ if $1 < f(x)$. }
   \end{cases}
\]

\begin{theorem}\label{thm:invariance}
 Suppose $f$ is a multilinear polynomial such that $f(\xi) \in [0, 1]$
 for all $\xi \in \{-1, 1\}^n$. If $f$ satisfies
 $\max_i \Inf_i (f) \le \tau$ then for
 any $\eta > 0$,
 \begin{equation}\label{eq:truncation}
   \E (f_\eta (X) - \overline{f_\eta} (X))^2 \le C \tau^{c\eta}
 \end{equation}
\end{theorem}

We will now use
Theorem~\ref{thm:invariance} to prove Theorem~\ref{thm:robust-mist}.
First,~\eqref{eq:truncation} and the triangle inequality imply that
for any $0 < \rho' < 1$,
\begin{equation}\label{eq:ns-gaussian}
 \E_{\rho'} f_\eta (X) f_\eta (Y) \le \E_{\rho'} \overline{f_\eta}(X) \overline{f_\eta}(Y)
 + C \tau^{c\eta}.
\end{equation}
Now, 
\begin{equation}\label{eq:ns-gaussian-to-boolean}
\E_{\rho'} f_\eta(X) f_\eta(Y)
= \E_{\rho'} f_\eta(\xi) f_\eta(\sigma)
= \E_{e^{-2\eta} \rho'} f(\xi) f(\sigma).
\end{equation}
If we set $\rho' = e^{2\eta} \rho$ (assuming that $\eta$ is small enough so that
$e^{2\eta} \rho < 1$) then~\eqref{eq:ns-gaussian},~\eqref{eq:ns-gaussian-to-boolean},
and the assumption (\ref{eq:robust-mist}) of Theorem~\ref{thm:robust-mist} imply that
\begin{align*}
 \E_{\rho'} \overline{f_\eta}(X) \overline{f_\eta}(Y)
 & \ge J(\E f, \E g; \rho) - C \tau^{c\eta} - \delta \\
 & \ge J(\E \overline{f_\eta}, \E \overline{f_\eta}; \rho) - C \tau^{c\eta} - \delta \\
 & \ge J(\E \overline{f_\eta}, \E \overline{f_\eta}; \rho') - C(\rho) \eta - C \tau^{c\eta} - \delta,
\end{align*}
where the second inequality follows because (by~\eqref{eq:truncation})
$|E f - \E \overline{f_\eta}| \le C \tau^{c^\eta}$ and $\pdiff{J(x,y; \rho)}{x}$
is bounded.
Applying Corollary~\ref{cor:1-func-robustness} (with $\rho'$ in place of $\rho$)
to $\overline{f_\eta}$, we
see that there are $a, b \in \R^n$ such that
\[
 \E (\overline{f_\eta}(X) - 1_{\{\inr{a}{X-b} \ge 0\}})^2 \le
 C(\rho) (\eta + \tau^{c\eta} + \delta)^c.
\]
By~\eqref{eq:truncation} and the triangle inequality, we may replace
$\overline{f_\eta}$ by $f_\eta$:
\begin{equation}\label{eq:close-to-gaussian-half-space}
 \E (f_\eta(X) - 1_{\{\inr{a}{X-b} \ge 0\}})^2 \le
 C(\rho) (\eta + \tau^{c\eta} + \delta)^c.
\end{equation}

The next step is to pull~\eqref{eq:close-to-gaussian-half-space} back to
the discrete cube by replacing $X$ by $\xi$. We will do this using
Theorem~\ref{thm:invariance}. As a prerequisite, we need to
show that $1_{\{\inr{a}{x-b} \ge 0\}}$ has small influences; this is
essentially the same as saying that $a$ is well-spread:
\begin{lemma}\label{lem:well-spread}
There is an $a \in \R^n$ satisfying~\eqref{eq:close-to-gaussian-half-space} with
$\sum a_i^2 = 1$ and $\max_i |a_i| \le C \tau^c$.
\end{lemma}

Once we have shown that $1_{\{\inr{a}{x-b} \ge 0\}}$ has small influences,
we can use Theorem~\ref{thm:invariance} to show that the multilinear extension
of $1_{\{\inr{a}{x-b} \ge 0\}}$ is close to $1_{\{\inr{a}{x-b} \ge 0\}}$:

\begin{lemma}\label{lem:half-space-extension}
 Let $g^{a,b}$ be the multilinear extension of the function
 $x \mapsto 1_{\{\inr{a}{x - b} \ge 0\}}$. If
 $\sum_i a_i^2 = 1$ and $\max_i |a_i| \le \tau$ then
 for any $\eta > 0$,
 \[
  \E (g_\eta^{a,b}(X) - 1_{\{\inr{a}{X - b} \ge 0\}})^2 \le C (\eta + \tau^{c\eta}).
 \]
\end{lemma}

From Lemma~\ref{lem:half-space-extension} and the triangle inequality, we
conclude from~\eqref{eq:close-to-gaussian-half-space} that
\[
 \E (f_\eta(X) - g_\eta^{a,b}(X))^2 \le
 C(\rho) (\eta + \tau^{c\eta} + \delta)^c.
\]
Since $f_\eta - g_\eta^{a,b}$ is a multilinear polynomial, its second moment
remains unchanged when $X$ is replaced by $\xi$:
\[
 \E (f_\eta(\xi) - g_\eta^{a,b}(\xi))^2 \le
 C(\rho) (\eta + \tau^{c\eta} + \delta)^c.
\]
Now, $g^{a,b}$ is the indicator of a half-space on the cube; thus,
$\E (g_\eta^{a,b}(\xi) - g^{a,b}(\xi))^2 \le C \eta^c$
(see, for example,~\cite{BeKaSc:99}). Applying this and the triangle inequality, we have
\begin{equation}\label{eq:f-eta-close}
 \E (f_\eta(\xi) - g^{a,b}(\xi))^2 \le
 C(\rho) (\eta + \tau^{c\eta} + \delta)^c.
\end{equation}
The last piece is to replace $f_\eta$ by $f$. We do this with a simple lemma which
shows that for any function $f$, if $f_\eta$ is close to some indicator function
then $f$ is also close to the same indicator function.
\begin{lemma}\label{lem:indicator-close}
 For any functions $f: \{-1, 1\}^n \to [0, 1]$ and $g: \{-1, 1\}^n \to \{0, 1\}$
 and any $\eta > 0$,
 \[
  \E (f(\xi) - g(\xi))^2 \le C \E (f_\eta(\xi) - g(\xi))^2.
 \]
\end{lemma}

Applying Lemma~\ref{lem:indicator-close} to~\eqref{eq:f-eta-close}, we obtain
\[
 \E (f(\xi) - g^{a,b}(\xi))^2 \le
 C(\rho) (\eta + \tau^{c\eta} + \delta)^c.
\]
By choosing $\tau$ and $\eta$ small enough compared to $\delta$,
the proof of Theorem~\ref{thm:robust-mist} is complete, modulo the proofs
of Lemmas~\ref{lem:well-spread},~\ref{lem:half-space-extension}
and~\ref{lem:indicator-close}. We will prove them in the coming section.

\subsection{Gaussian and boolean half-spaces}

Here we will prove the lemmas of the previous section. Before
doing so, let us observe that
$\E X_i 1_{\{\inr{a}{X-b} \ge 0\}}$
is proportional to $a_i$, a fact which has already been noted
by Matulef et al.~\cite{MORS:09}:

\begin{lemma}\label{lem:half-space-corr}
\[\E X_i 1_{\{\inr{a}{X-b} \ge 0\}} =
 a_i \phi(\inr{a}{b}).
 \]
\end{lemma}

\begin{proof}
Let $e_i \in \R^n$ be the vector with 1 in position $i$ and 0 elsewhere. 
We may write $e_i = a_i a + a^\perp$, where $a^\perp$ is some element
of $\R^n$ which is orthogonal to $a$. Note that $\inr{X}{a^\perp}$ is
independent of $\inr{X}{a}$ and so
$\E \inr{X}{a^\perp} 1_{\{\inr{a}{X-b} \ge 0\}} = 0$. Hence,
\[
 \E X_i 1_{\{\inr{a}{X-b} \ge 0\}}
 = a_i \E \inr{a}{X} 1_{\{\inr{a}{X-b} \ge 0\}}
 = a_i \E X_1 1_{\{X_1 \ge \inr{a}{b}\}},
\]
where the second inequality follows because, by the rotational invariance
of the Gaussian measure, $\inr{a}{X}$ has the same distribution as $X_1$.
Finally, integration by parts implies that $\E X_1 1_{\{X_1 \ge \inr{a}{b}\}}
= \phi(\inr{a}{b})$.
\end{proof}

Next, we prove Lemma~\ref{lem:well-spread}. The point is that if a half-space
is close to a low-influence function $f$ then that half-space must also have low
influences. We can then perturb the half-space to have even lower influences
without increasing its distance to $f$ by much.

\begin{proof}[Proof of Lemma~\ref{lem:well-spread}]
Suppose that $f$ has influences bounded by $\tau$, and that
\begin{equation}\label{eq:f-and-half-space-close}
 \E (f(X) - 1_{\{\inr{a}{X-b} \ge 0\}})^2 \le \gamma,
\end{equation}
where $\gamma = C(\rho)(\eta + \tau^{c\eta} + \epsilon)^c$.
We will show that there is some $\bar a$ such that
$\sum_i \bar a_i^2 = 1$,
$\max_i |\bar a_i| \le C \tau^c$, and
\begin{equation}\label{eq:lem-well-spread-goal}
 \E (f(X) - 1_{\{\inr{\bar a}{X-b} \ge 0\}})^2 \le \gamma^c.
\end{equation}
When applied to the function $f_\eta$,
this will imply the claim of Lemma~\ref{lem:well-spread}.

Since the influences of $f$ are bounded by $\tau$, it follows in particular
that $|\hat f_{\{i\}}| \le \tau$ for every $i$. On the other hand,
$X_i$ form an orthonormal sequence and so
\begin{align}
 \E (f(X) - 1_{\{\inr{a}{X-b} \ge 0\}})^2
 &\ge \sum_i \big(\E X_i f(X) - \E X_i 1_{\{\inr{a}{X-b} \ge 0\}}\big)^2 \notag \\
 &= \sum_i \big(\hat f_{\{i\}} - a_i \phi(\inr{a}{b})\big)^2,
 \label{eq:f-and-half-space-far}
\end{align}
where the equality used Lemma~\ref{lem:half-space-corr}.
Defining $\kappa_{a,b} = \phi(\inr{a}{b})$,
it follows that for any $i$ with $|a_i| \kappa_{a,b} \ge C\tau$,
we have $\big(\hat f_{\{i\}} - a_i \kappa_{a,b}\big)^2 \ge
c a_i^2 \kappa_{a,b}^2$.
Combining this with~\eqref{eq:f-and-half-space-close} and~\eqref{eq:f-and-half-space-far},
\begin{equation}\label{eq:high-weights-ai}
 \gamma \ge \E (f(X) - 1_{\{\inr{a}{X-b} \ge 0\}})^2
 \ge c \kappa_{a,b}^2 \sum_{i: |a_i| \kappa_{a,b} \ge C \tau}  a_i^2.
\end{equation}
for every $i$.

We consider two cases, depending on whether $\kappa_{a,b}$ is large
or small. First, suppose that $\kappa_{a,b} \le \gamma^{1/3}$. Now,
$\kappa_{a,b} \ge c \Pr(X_1 \ge \inr{a}{b}) = \E 1_{\{\inr{a}{X-b} \ge 0\}}$,
while~\eqref{eq:f-and-half-space-close} implies that
\[\E f \le \sqrt{\gamma} + \E 1_{\{\inr{a}{X-b} \ge 0\}} \le \sqrt{\gamma}
+ C \gamma^{1/3} \le C \gamma^{1/3}.
\]
Since $f$ takes values in $[0, 1]$, it follows that $f$ is close to the zero function;
in particular,\emph{any} half-space with small enough measure will
satisfy~\eqref{eq:lem-well-spread-goal}.

Now suppose that $\kappa_{a,b} \ge \gamma^{1/3}$ (which is in turn larger than 
$\tau^{1/3}$); then~\eqref{eq:high-weights-ai} implies that
\[
 \sum_{i: |a_i| \ge C \tau^{2/3}}  a_i^2 \le C \gamma^{1/3}.
\]
If we define $\bar a$ to be the truncated version of $a$ (i.e.\ $\bar a_i = a_i$
if $|a_i| < C \tau^{2/3}$ and $\bar a_i = 0$ otherwise), then this implies
that $|a - \bar a|^2 \le C \gamma^{1/3}$. Moreover, $|\bar a|^2 \ge 1 - C \gamma^{1/3}$,
which implies that we can normalize $\bar a$ so that $|\bar a| = 1$, while preserving
the fact that $|a - \bar a|^2 \le C \gamma^{1/3}$ and $\max_i \bar a_i \le C 
\gamma^{1/3}$. Finally, $|a - \bar a|^2 \le C\gamma^{1/3}$ implies that
\[
 \E (1_{\{\inr{a}{X-b}\}} - 1_{\{\inr{\bar a}{X-b}\}})^2 \le C \gamma^c.
\]
By the triangle inequality and~\eqref{eq:f-and-half-space-close},~\eqref{eq:lem-well-spread-goal} follows.
\end{proof}

Next, we will prove Lemma~\ref{lem:half-space-extension}: if $g^{a,b}$ is the linear
extension of a low-influence half-space, then $g^{a,b}$ is close to a half-space.
Observe that this is very much not the case for general half-spaces: the linear
extension of $1_{x_1 \ge 0}$ is $x_1$, which is not close, in
$L_2(\R^n, \gamma_n)$, to any half-space.

\begin{proof}[Proof of Lemma~\ref{lem:half-space-extension}]
The proof rests on the invariance principle (Theorem~\ref{thm:invariance}).
Let $g$ be the linear extension of $1_{\{\inr{a}{x-b} \ge 0\}}$ and
let $h(x) = \inr{a}{x - b}$.
First of all, the Berry-Esseen theorem implies that
for any $M > 0$,
\begin{align*}
 \E g(\xi) h(\xi)
 &= \E h(\xi) 1_{h(\xi) \ge 0} \\
 &= \int_{\inr{a}{b}}^\infty \Pr(\inr{a}{x} \ge t)\ dt \\
 &\ge \int_{\inr{a}{b}}^M \Pr(\inr{a}{x} \ge t)\ dt \\
 &\ge \int_{\inr{a}{b}}^M \Pr(X_1 \ge t)\ dt - C M \tau \\
 &\ge \int_{\inr{a}{b}}^\infty \Pr(X_1 \ge t)\ dt - C M \tau - C e^{-M^2/2}.
\end{align*}
Choosing $M = \sqrt{\log(1/\tau)}$, we have
\begin{equation}\label{eq:boolean-cov-gaussian-cov}
 \E g(\xi) h(\xi) \ge \int_{\inr{a}{b}}^\infty \Pr(X_1 \ge t)\ dt - C \tau^c
 = \E 1_{\{\inr{a}{X-b} \ge 0\}} h(X) - C \tau^c.
\end{equation}
Now, $h$ is linear and so $h_t = e^{-t} h$; since $Q_\eta$ is self-adjoint, we have
\begin{align*}
 \E g(\xi) h(\xi)
 &= e^{\eta} \E g_\eta(\xi) h(\xi) \\
 &= e^{\eta} \E g_\eta(X) h(X) \\
 &\le e^{\eta} \E \overline{g_\eta}(X) h(X) + C e^\eta (\eta + \tau^{c\eta}) \\
 &\le \E \overline{g_\eta}(X) h(X) + C (\eta + \tau^{c\eta}),
\end{align*}
where the last inequality assumes that $\eta < 1$ (if not
then the lemma is trivial anyway).
Combining this with~\eqref{eq:boolean-cov-gaussian-cov},
\begin{equation}\label{eq:inner-prod-with-indicator}
\E 1_{\{\inr{a}{X-b} \ge 0\}} h(X) \le \E \overline{g_\eta}(X) h(X)
+ C(\eta + \tau^{c\eta}).
\end{equation}

Now, let $m(X) = 1_{\{\inr{a}{X-b} \ge 0\}} - \overline {g_\eta}(X)$
and take $\epsilon = \E |m|$;
note when $m \ne 0$ then $m$ and $h$ have the same sign.
Let $A = \{x: \inr{a}{X-b} \in [-\epsilon/2, \epsilon/2]\}$. Then
$\Pr(A) \le \epsilon/2$, and since $|m| \le 1$ we must have
$\E |m| 1_{A^c} \ge \E |m| - \Pr(A) \ge \epsilon/2$.
But on $A^c$ we have $|h(x)| \ge \epsilon/2$;
since the signs of $m$ and $h$ agree,
\[
 \E m(X) h(X) \ge \E m(X) h(X) 1_{\{X \in A^c\}} \ge \frac{\epsilon}{2}
 \E |m| 1_{A^c} \ge \frac{\epsilon^2}{4}.
\]
Applying this to~\eqref{eq:inner-prod-with-indicator} yields
$\epsilon
\le C(\eta + \tau^{c\eta})^c$.
So if we recall the definition of $\epsilon$, then we see that
\[
\E |1_{\{\inr{a}{X-b} \ge 0\}} - \overline {g_\eta}(X)|
\le C(\eta e^{2\eta} + \tau^{c\eta})^c.
\]
By changing the constant $c$, we may replace $\E |\cdot|$ with $\E(\cdot)^2$;
by~\eqref{eq:truncation}, we may replace $\overline{g_{\eta}}$ by
$g_\eta$. This completes the proof of the lemma. Note that the only
reason for proving this lemma with $g_\eta$ instead of $g$ was for extra convenience
when applying it; the statement of the lemma is also true with $g$ instead
of $g_\eta$.
\end{proof}

The only remaining piece is Lemma~\ref{lem:indicator-close}.

\begin{proof}[Proof of Lemma~\ref{lem:indicator-close}]
Suppose $f: \{-1, 1\}^n \to [-1, 1]$ and $g: \{-1, 1\}^n \to \{-1, 1\}$.
This does not exactly correspond to the statement of the lemma, but it will
be more convenient for the proof; we can recover the statement of the lemma
by replacing $f$ by $\frac{1 + f}{2}$ and $g$ by $\frac{1 + g}{2}$.

Let $\epsilon = \E (f_\eta (\xi) - g(\xi))^2$.
Since $g$ takes values in $\{-1, 1\}$, we have $\E g^2 = 1$; then the triangle
inequality implies that
\[
 \E f_\eta^2 \ge \E g^2 - 2\epsilon = 1 - 2\epsilon.
\]
Since $\E f^2 \le 1$, we have
\begin{align*}
 \E (f - f_\eta)^2
 &= \sum_{S \subset [n]} \hat f_S^2 (1 - e^{-\eta |S|})^2 \\
 &\le \sum_{S \subset [n]} \hat f_S^2 (1 - e^{-\eta |S|}) \\
 &= \E f^2 - \E f_\eta^2 \\
 &\le 2\epsilon.
\end{align*}
It then follows by the triangle inequality that
$\E (f - g)^2 \le C \epsilon$.
\end{proof}

\section{Spherical noise stability}
 We now use Theorem~\ref{thm:2-funcs-robustness} to prove Theorem~\ref{thm:sphere}. For a subset $A \subset S^{n-1}$, we define $\bar{A} \subset \R^n$ to be the {\em radial 
  extension} of $A$:
  \[
  \bar{A} = \{x \in \R^n: x \ne 0 \text{ and }
  \frac{x}{|x|} \in A\}
  \]
  From the spherical symmetry of the Gaussian distribution it immediately follows that $\P(\bar{A}) = Q(A)$. 
  The proof of Theorem~\ref{thm:sphere} crucially relies on the fact that $Q_{\rho}(A_1,A_2)$ is close to 
  $\Pr_{\rho}(\bar{A_1},\bar{A_2})$ in high dimensions. More explicitly it uses the following lemmas:
  
    \begin{lemma} \label{lem:cap-half-space}
   For any half-space $H = \{x \in \R^n : \inr{a}{x} \le b\}$ there is a spherical cap $B = \{x \in S^{n-1} : \inr{a}{x} \le b'\}$ such that $\Pr(\bar{B}) = \Pr(H)$ and 
   \[
    \P(\bar{B} \symdiff H) \le C n^{-1/2} \log n.
   \]
  \end{lemma}

  \begin{lemma} \label{lem:sphere} 
  For any two sets $A_1, A_2 \subset S^{n-1}$ and any $\rho \in [-1+\epsilon,1-\epsilon]$ it holds that 
  \[
  |Q_{\rho}(A_1,A_2) - \Pr_{\rho}(\bar{A_1},\bar{A_2})| \leq C(\eps) n^{-1/2} \log n.
  \]
  \end{lemma}

  Given Lemmas~\ref{lem:sphere} and~\ref{lem:cap-half-space}, the proof of Theorem~\ref{thm:sphere} is an easy corollary of Theorem~\ref{thm:2-funcs-robustness}:
  
  \begin{proof}[Proof of Theorem~\ref{thm:sphere}]
  Define $\delta_\ast = \delta(\bar{A_1}, \bar{A_2})$.
  Let $H_1, H_2$ be parallel half-spa\-ces with $\Pr(H_i) = \Pr(\bar{A_i})$, and
  let $B_1,B_2$ be the corresponding caps whose existence
  is guaranteed by Lemma~\ref{lem:cap-half-space}. 
  Then 
  \begin{align*}
 \delta_{\ast} &= \delta(\bar{A_1}, \bar{A_2}) \\
  &= \P_\rho(X \in H_1, Y \in  H_2) - \P_{\rho}(X \in \bar{A_1}, Y \in \bar{A_2}) \\
  &\le \P_\rho(X \in \bar{B_1}, Y \in \bar{B_2}) - \P_{\rho}(X \in \bar{A_1}, Y \in \bar{A_2}) + O(n^{-1/2} \log n) \\
  &\le Q_\rho(X \in B_1, Y \in B_2) - Q_{\rho}(X \in A_1, Y \in A_2) + O(n^{-1/2}\log n) \\
  & = \delta(A_1,A_2) + O(n^{-1/2} \log n),
 \end{align*}
 where the first inequality follows from Lemma~\ref{lem:cap-half-space}
 and the second follows from Lemma~\ref{lem:sphere}.
 
 From Theorem~\ref{thm:2-funcs-robustness} it follows that there are parallel half-spaces $H_1$ and $H_2$ with $\P(H_i) = \P(\bar{A_i})$ satisfying
 \[
 \Pr(\bar{A_i} \symdiff H_i)
    \le C(\rho) m^{c(\rho)} \delta_\ast^{\frac{1}{4} \frac{(1 - \rho)(1-\rho^2)}{1 + 3\rho}}.
 \]
 By Lemma~\ref{lem:cap-half-space}, there are parallel caps
 $B_1$ and $B_2$ such that
  \[
 Q(A_i \symdiff B_i) = \P(\bar{A_i} \symdiff \bar{B_i})
    \le C(\rho) m^{c(\rho)} \delta_\ast^{\frac{1}{4} \frac{(1 - \rho)(1-\rho^2)}{1+3\rho}}.
    \qedhere
 \]
 \end{proof}
 
 The proof of Lemma~\ref{lem:cap-half-space} is quite simple, so we present
 it first:
 \begin{proof}[Proof of Lemma~\ref{lem:cap-half-space}]
   Let $H = \{x \in \R^n: \inr{a}{x} \le b\}$, and suppose without loss
   of generality that $b \ge 0$. For any $\epsilon > 0$, define
   \begin{align*}
    H_\epsilon^+ &= \{x \in \R^n: \inr{a}{x} \le b(1 + \epsilon)\} \\
    H_\epsilon^- &= \{x \in \R^n: \inr{a}{x} \le b(1 - \epsilon)\}.
   \end{align*}
   Note that $\Pr(H_\epsilon^+ \setminus H_\epsilon^-) \le C \epsilon$.
   
   Now define $B = \{x \in S^{n-1}: \inr{x}{a} \le b /\sqrt n\}$. Then
   $\bar B = \{x \in \R^n : \inr{x}{a} \le b |x| /\sqrt n\}$, and so
   \begin{align*}
    \Pr(\bar B \setminus H_\epsilon^+)
    &= \Pr((1 + \epsilon) b \le \inr{X}{a} \le b |X| / \sqrt n) \\
    &\le \Pr(|X| \ge (1 + \epsilon) \sqrt n) \\
    &\le C e^{-c\epsilon^2 n},
   \end{align*}
  where the last line follows from standard concentration inequalities
  (Bernstein's inequalities, for example). Similarly,
  \[
   \Pr(H_\epsilon^- \setminus \bar B)
   \le \Pr(|X| \le (1 - \epsilon) \sqrt n) \le C e^{-c \epsilon^2 n}.
  \]
  Since $H_\epsilon^- \subset H \subset H_\epsilon^+$ and
  $\Pr(H_\epsilon^+ \setminus H_\epsilon^-) \le C \epsilon$, it follows that
  \[
   \Pr(H \symdiff \bar B) \le C \epsilon + C e^{-c \epsilon^2 n}.
  \]
  By choosing $\epsilon = C n^{-1/2} \log n$, we have
  \begin{equation}\label{eq:small-symdiff}
   \Pr(H \symdiff \bar B) \le C n^{-1/2} \log n.
   \end{equation}
   
   Now, the lemma claimed that we could ensure $\Pr(\bar B) = \Pr(H)$.
   Since the volume of the cap $B' := \{\inr{a}{x} \le b'|x|\}$ is continuous
   and strictly increasing in $b'$, we may define $b'$ to be the unique
   real number such that $\Pr(\bar B') = \Pr(H)$.
   Now, either $B \subset B'$ or $B' \subset B$; hence
   $\Pr(\bar B \symdiff \bar B') = |\Pr(\bar B) - \Pr(\bar B')|$.
   On the other hand,~\eqref{eq:small-symdiff} implies
   that
   \[
    |\Pr(\bar B) - \Pr(\bar B')| = |\Pr(\bar B) - \Pr(H)| \le C n^{-1/2} \log n,
   \]
    and so
   the triangle inequality leaves us with
   \[
    \Pr(H \symdiff \bar B')
    \le \Pr(H \symdiff \bar B) + \Pr(B \symdiff \bar B')
    \le C n^{-1/2} \log n.\qedhere
   \]
 \end{proof}

 We defer the proof of Lemma~\ref{lem:sphere} until the next section,
 since this proof requires an introduction to spherical harmonics.
 
 \subsection{Spherical harmonics and Lemma~\ref{lem:sphere}}
 We will try to give an introduction to spherical harmonics which is
 as brief as possible, while still containing enough material for us to
 explain the proof of Lemma~\ref{lem:sphere} adequately. A slightly less
 brief introduction is contained in~\cite{KlartagRegev:11}; for a full
 treatment, see~\cite{Muller:66}.
 
 Let $\calS_k$ be the linear space consisting of harmonic,
 homogeneous, degree-$k$ polynomials.
 We will think of $\calS_k$ as a subspace of $L_2(S^{n-1}, Q)$; then
 $\{\calS_k: k \ge 0\}$ spans $L_2(S^{n-1}, Q)$. One can easily check
 that $\calS_k$ is invariant under rotations. Hence it is a
 representation of $SO(n)$. It turns out, moreover, that $\calS_k$
 is an irreducible representation of $SO(n)$; combined with
 Schur's lemma, this leads to the following important property:
 \begin{lemma}\label{lem:irreducible-eigenspaces}
  If $T: L_2(S^{n-1}) \to L_2(S^{n-1})$ commutes with rotations
  then $\{\calS_k: k \ge 0\}$ are the eigenspaces of $T$.
 \end{lemma}
 
 In particular, we will apply Lemma~\ref{lem:irreducible-eigenspaces}
 to the operators $T_\rho$ defined by
 $(T_\rho f)(X) = \E (f(Y) | X)$, where $(X, Y) \sim Q_\rho$. In other words,
 $(T_\rho f)(x)$ is the average of $f$ over the set
 $\{y \in S^{n-1} : \inr{x}{y} = \rho\}$. Clearly, $T_\rho$ commutes
 with rotations;
 hence Lemma~\ref{lem:irreducible-eigenspaces} implies that
 $\{\calS_k: k \ge 0\}$ are the eigenspaces of $T_\rho$. In particular,
 there exist $\{ \mu_k(\rho): k \ge 0\}$ such that $T_\rho f = \mu_k(\rho) f$
 for all $f \in \calS_k$. Moreover, to compute
 $\mu_k(\rho)$, it is enough to compute $T_\rho f$ for a single
 $f \in \calS_k$. For this task, the Gegenbauer polynomials provide good
 candidates: define
 \[
  G_k(t) = \E (t + i W_1 \sqrt{1-t^2})^k,
 \]
 where the expectation is over $W = (W_1, \dots, W_{n-1})$ distributed
 uniformly on the sphere $S^{n-2}$.
 Define $f_k(x) = G_k (x_1)$; it turns out that $f_k \in \calS_k$; on
 the other hand, one can easily check that $f_k(e_1) = 1$, while
 $(T_\rho f_k)(e_1) = G_k(\rho)$. From the discussion above, it then follows that
 \[
  \mu_k(\rho) = \E (\rho + i W_1 \sqrt{1 - \rho^2})^k.
 \]
 With this explicit formula, we can show that $\mu_k(\rho)$ is continuous in
 $\rho$:
 \begin{lemma}\label{lem:mu-k-cont}
 For any $\epsilon > 0$ there exists $C(\epsilon)$ such that
 if $\rho, \eta \in [-1+\epsilon, 1-\epsilon]$ then
 \[
  |\mu_k(\rho) - \mu_k(\eta)| \le C(\epsilon)(|\rho-\eta| + n^{-1/2}).
 \]
 \end{lemma}
 
 We will leave the proof of Lemma~\ref{lem:mu-k-cont} to the end. Instead,
 let us show how it can be used to prove that
 $Q_\rho(X \in A_1, Y \in A_2)$ is continuous in $\rho$.
 
 \begin{lemma}\label{lem:sphere_cont}
 For any $\epsilon > 0$ there exists $C(\epsilon)$ such that
 if $\rho, \eta \in [-1+\epsilon, 1-\epsilon]$ then
 \[
 |Q_\rho(X \in A_1, Y \in A_2) - Q_\eta(X \in A_1, Y \in A_2)| \leq \\
  C(\eps) Q^{1/2}(A_1) Q^{1/2}(A_2) (|\rho-\eta|  + n^{-1/2}).
  \]
 \end{lemma}
 
\begin{proof} 
Take $f, g \in L_2(S^{n-1}, Q)$ and
write $f = \sum_{k=0}^{\infty} f_k$ where $f_k \in \calS_k$.
Then 
\[
|\E g T_\rho f - \E g T_\eta f| \leq \| T_\rho f - T_\eta f \|_2 \| g \|_2
\]
(where $\|f\|_2$ denotes $\sqrt{\E f^2}$)
and 
\[
\| T_\rho f - T_\eta f \|_2^2 = \sum_{k=0}^{\infty} (\mu_k(\rho)-\mu_k(\eta))^2 \| f_k \|_2^2 
\]
By Lemma~\ref{lem:mu-k-cont}, we have
\[
\| T_\rho f - T_\eta f \|_2 \leq C(\eps)\big( |\rho-\eta| + n^{-1/2}\big) \| f \|_2, 
\]
and therefore 
\[
|\E g T_\rho f - \E g T_\eta f| \leq C(\eps) \| f \|_2 \| g \|_2 (|\rho-\eta| + n^{-1/2}).
\]
Note that if $f = 1_{A_1}$ and $g = 1_{A_2}$ then
$\E g T_\rho f = Q_\rho (X \in A_1, Y \in A_2)$, while
$\|f\|_2 = Q(A_1)^{1/2}$.
\end{proof}

 The proof of Lemma~\ref{lem:sphere} is straightforward once we know
 Lemma~\ref{lem:sphere_cont}. As we have already mentioned, normalized
 Gaussian vectors from $\Pr_\rho$ have a joint distribution that is
 similar to $Q_\rho$, except that their inner products are close to $\rho$
 instead of being exactly $\rho$. But Lemma~\ref{lem:sphere_cont} implies
 that a small difference in $\rho$ doesn't affect the noise sensitivity by
 much.
 
 \begin{proof}[Proof of Lemma~\ref{lem:sphere}]
 Let $X,Y$ be distributed according to $\P_\rho$. Then
 \[
 \P_{\rho}(X \in \bar{A_1}, Y \in \bar{A_2})  = \P_{\rho}\Big(\frac{X}{|X|} \in A_1, \frac{Y}{|Y|} \in A_2\Big),
 \]
 Note that conditioned on $|X|, |Y|$ and $\inr{X}{Y}$, the variables $X/|X|, Y/|Y|$ are distributed according to 
 $Q_r$, where $r=\inr{X}{Y}/(|X| | Y |)$. 
 Now with probability $1-\frac{1}{n^2}$ it holds that 
 \[
 | X |^2, | Y |^2 \in n \pm C n^{1/2} \log n, \quad 
 \inr{X}{Y} \in \rho n \pm C n^{1/2} \log n.
 \]
 On this event, we have
 \[
 r = \Big\langle\frac{X}{|X|},\frac{Y}{|Y|}\Big\rangle \in \rho \pm C n^{-1/2} \log n.
 \]
 Using Lemma~\ref{lem:sphere_cont} we get that 
 \[
 \P_{\rho}(X \in \bar{A_1}, Y \in \bar{A_2}) \leq Q_{\rho}(X \in A_1, Y \in A_2) + C(\epsilon) n^{-1/2} \log n.
 \]
A similar argument yields a bound in the other direction and concludes the proof. 
\end{proof} 

Our final task is the proof of Lemma~\ref{lem:mu-k-cont}:
\begin{proof}[Proof of Lemma~\ref{lem:mu-k-cont}]
Define $Z_\rho = \rho + iW_1 \sqrt{1-\rho^2}$
(recalling that
$W = (W_1, \dots, W_{n-1})$ is uniformly distributed
on $S^{n-2}$)
so that $\mu_k(\rho) = \E Z_\rho^k$.
Note that if $|W_1| \le \frac 12$ (which happens with probability at least $1 - \exp(-cn)$) then
\[
|Z_\rho| = \rho^2+ W_1 (1-\rho^2) \le \frac{1 + \rho^2}{2} \le 1 - \frac{\epsilon}{2}
\le \exp(-c\epsilon).
\]

Now,
\begin{align}
\mu_k(\rho)-\mu_k(\eta)
&= \E(Z_\rho^k - Z_\eta^k) \notag \\
&= \E (Z_\rho-Z_\eta) \sum_{j=1}^{k-1} Z_\rho^j Z_\eta^{k-1-j}.
\label{eq:eigen-1}
\end{align}
If $|W_1| \le \frac 12$ then $|Z_\rho^j Z_\eta^{k-1-j}| \le \exp(-c \eps k)$
and so 
\[\Big|\sum_j Z_\rho^j Z_\eta^{k-1-j}\Big| \le k \exp(-c \eps k) \le C(\eps)\]
Applying this to~\eqref{eq:eigen-1}, we have
\begin{align*}
 |\mu_k(\rho) - \mu_k(\eta)| &= \E (Z_\rho^k - Z_\eta^k) 1_{\{|W_1| \ge 1/2\}} + 
 \E 1_{\{|W_1| < 1/2\}} (Z_\rho-Z_\eta) \sum_{j=1}^{k-1} Z_\rho^j Z_\eta^{k-1-j} \\ 
 &\le 2 \Pr(|W_1| \ge 1/2) + C(\eps) \E |Z_\rho - Z_\eta| \\
 &\le \exp(-cn) + C(\eps) |\rho - \eta|,
\end{align*}
where $\E |Z_\rho - Z_\eta| \le C(\eps) |\rho - \eta|$ because
$|\sqrt{1 - \rho^2} - \sqrt{1 - \eta^2}| \le C(\eps)|\rho - \eta|$.
\end{proof}

\subsection{Spherical noise and Max-Cut}

In this section, we will outline how robust noise sensitivity
on the sphere (Theorem~\ref{thm:sphere})
implies that half-space rounding for the Goemans-Williamson algorithm
is robustly optimal (Theorem~\ref{thm:optimal-rounding}).
The key for making this connection is Karloff's family of graphs~\cite{Karloff:99}:
for any $n, d \in \N$, let $G_{n,d} = (V_{n,d}, E_{n,d})$ be the graph whose vertices are the $\binom{n}{n/2}$
balanced elements of $\{-n^{-1/2}, n^{-1/2}\}^n$, and with an edge between $u$ and $v$
if $\inr{u}{v} = d/n$. Karloff showed that
if $d \le n/24$ then
the optimal cut of $G_{n,d}$ has value $|E_{n,d}| (1-d/n)$.
Moreover, the identity
embedding (and any rotation of it) is an optimal embedding of $G_{n,d}$ into $S^{n-1}$.
In these embeddings, every angle between two connected vertices is $d/n$; hence, it is easy
to calculate the expected value of a rounding scheme:

\begin{lemma}\label{lem:cut-noise}
 Let $(X, Y)$ be distributed according to $Q_{d/n}$.
 For any rounding scheme $R$,
 \[
 \Cut(G_{n,d}, R) \le \frac{|E_{n,d}|}{2} \E |R(X) - R(Y)|,
 \]
 where the expectation is with respect to $X, Y$ and $R$.
\end{lemma}

\begin{proof}
Recall that
\begin{align*}
 \Cut(G, R) &= \frac{1}{2} \min_f \E_R \sum_{(u, v) \in E} |R(f(u)) - R(f(v))| \\
 &\le \frac{1}{2} \E_R \E_f \sum_{(u, v) \in E} |R(f(u)) - R(f(v))|,
\end{align*}
where the expectation is taken over all rotations $f$. But if $f$ is a uniformly random rotation
then for every $(u, v) \in E_{n,d}$, the pair $(f(u), f(v))$
is equal in distribution to the pair $(X, Y)$ (and both pairs are independent of $R$).
\end{proof}

Theorem~\ref{thm:optimal-rounding} follows fairly easily from
Lemma~\ref{lem:cut-noise}, Theorem~\ref{thm:sphere},
and the fact that $\MaxCut(G_{n,d}) = |E_{n,d}| (1 - d/n)$.
Indeed, choose $n$ and $d$ such that
$|d/n - \cos^{-1} \theta^*| \le n^{-1}$, where $\theta^* \approx 2.33$ minimizes $\alpha_\theta$,
and suppose there is a rounding scheme $R$ such that
\[
 \Cut(G_{n,d}, R) \ge \MaxCut(G_{n,d})(\alpha_{\theta^*} - \epsilon).
\]
Let $\theta = \cos(d/n)$;
since $\alpha_\theta$ is continuous in $\theta$, it follows that
$|\alpha_\theta - \alpha_{\theta^*}| \le \frac{C}{n}$. Taking
$\epsilon_\star = \max\{\epsilon, n^{-1/2} \log n\}$, we have
$|\alpha_\theta - \alpha_{\theta^*}| \le C \epsilon_\star$ and so
\begin{align*}
 \Cut(G_{n,d}, R)
 &\ge \MaxCut(G_{n,d})(\alpha_{\theta} - C \epsilon_\star) \\
 &= |E_{n,d}|(1-\cos \theta)(\alpha_{\theta} - C \epsilon_\star) \\
 &= \frac 2\pi \theta |E_{n,d}| (1-C \epsilon_\star).
\end{align*}
By Lemma~\ref{lem:cut-noise},
$
 \frac 12 \E |R(X) - R(Y)| \ge \frac 2\pi \theta (1-C \epsilon_\star).
$
If we define the (random) subset $A_R \subset S^{n-1}$ by $A_R = \{x : R(x) = 1\}$,
and set $\rho = \cos \theta$ then
\[
 \Pr(A_R) - \Stab_\rho(A_R) = \frac 12 \E \big(|R(X) - R(Y)|\big\mid R\big)
\]
Taking expectations,
\begin{equation}\label{eq:cut-stability}
 \E(\Pr(A_R) - \Stab_\rho(A_R)) = \frac 12 \E |R(X) - R(Y)|
 \ge \frac 2\pi \arccos \rho - C \epsilon_\star.
\end{equation}
Let $\delta_R$ be the random deficit
$\delta_R = \frac 2\pi \arccos \rho - (\Pr(A_R) - \Stab_\rho(A_R))$, so
that~\eqref{eq:cut-stability} implies $\E \delta_R \le C \epsilon_\star$.
Take $\eta_R$ to be the distance from $A_R$ to the nearest hemisphere:
$\eta_R = \min\{\Pr(A_R \symdiff B) : B \text{ is a hemisphere}\}$ and let
$B_R$ be a hemisphere that achieves the minimum (which is attained because the set
of hemispheres is compact with respect to the distance $d(A, B) = \Pr(A \symdiff B)$).
Recall that $\theta \approx \theta^* \approx 2.33$ and so $\rho = \cos \theta < 0$;
by the same symmetries discussed following Theorem~\ref{thm:2-funcs-robustness},
Theorem~\ref{thm:sphere} applies for $\rho < 0$, but with the deficit inequality
reversed. Hence, $\eta_R \le C \max\{\delta_R, n^{-1/2} \log n\}^c$.
Taking expectations,
\[
 \E \eta_R
 \le C \E \max\{\delta_R, n^{-1/2}\log n\}^c
 \le C \max\{\E \delta_R, n^{-1/2}\log n\}^c = C' \epsilon_\star^c.
\]

Consider the rounding scheme $\tilde R(y)$ which is 1 when $y \in B_R$ and $-1$ otherwise.
Then $\E \big(|R(Y) - \tilde R(Y)|\big\mid R\big) = 2 \eta_R$, and so the displayed equation
above implies that
\[\E |R(Y) - \tilde R(Y)| \le C \epsilon_\star^c.\]
Since $\tilde R$ is a hyperplane rounding scheme, this completes the proof of Theorem~\ref{thm:optimal-rounding}.

\subsection*{Acknowledgements}
Part of the work on this paper
was done while the second author was visiting the
Universit\'e de Paul Sabatier. He would like to thank
Michel Ledoux and Franck Barthe for hosting him, and for fruitful
discussions.

\bibliography{my,robustness,all}
\bibliographystyle{plain}
\end{document}